\documentclass[12pt]{amsart}
\usepackage{a4wide}
\usepackage{amssymb}
\usepackage[active]{srcltx}
\usepackage{epsfig}
\usepackage{amscd}
\usepackage[usenames]{color}

\newcommand{\Aff}{{\text{Aff}}}
\newcommand{\Aut}{{\text{Aut}}}
\newcommand{\SL}{{\text{SL}}}
\newcommand{\GL}{{\text{GL}}}
\newcommand{\PGL}{{\text{PGL}}}
\newcommand{\PM}{{\text{PM}}}

\newcommand{\br}{{\mathbb R}}

\newcommand{\RP}{{\mathbb{{RP}}}}
\newcommand{\ra}{{\rightarrow}}
\newcommand{\rp}[1]{\mathbb {RP}^{#1}}
\newtheorem{theorem}{Theorem}[section]
\newtheorem{Prop}[theorem]{Proposition}
\newtheorem{lemma}[theorem]{Lemma}
\newtheorem{co}[theorem]{Corollary}
\numberwithin{equation}{section}
\DeclareMathOperator{\AC}{AC} 

\theoremstyle{definition}
\newtheorem{Def}[theorem]{Definition}
\newtheorem{remark}[theorem]{Remark}

\begin{document}
\title[Markus Conjecture]
{On the Markus conjecture in Convex case}
\author{Kyeonghee Jo and Inkang Kim}
\subjclass[2010]{57N16, 52A20} \keywords{Markus Conjecture, parallel volume, quasi-homogeneous, homogeneous, limit set}

\date{}
\address{Division of Liberal Arts and sciences,
  Mokpo National Maritime University, Mokpo, Chonnam, 58628, Korea  }
\email{khjo@mmu.ac.kr}
\address{School of Mathematics, Korea Institute for Advanced Study, Seoul, 02455, Korea }
\email{inkang@kias.re.kr}
\thanks {The second author gratefully acknowledges the partial support of  Grant NRF-2017R1A2A2A05001002.}
\maketitle

\begin{abstract}
In this paper, we show that any convex affine domain with a nonempty 
limit sets on the boundary under the action of the identity
component of the automorphism group cannot cover a compact affine
manifold with a parallel volume, which is a positive answer to the Markus conjecture for convex case. Consequently, we  show that the Markus conjecture is true for convex affine manifolds of  dimension  $\leq 5$.  
\end{abstract}

\section{Introduction}
A topological manifold $M$ can be equipped with a $(G,X)$-structure where $X$ is a model space and $G$ is a group acting on $X$ so that $M$ has an atlas $(\phi_i, U_i)$ from
open sets $U_i$ in $M$ into open sets in $X$ and the transition maps $\phi_i\circ \phi_j^{-1}$ are restrictions of  elements in $G$.
Depending on the choice of $(G,X)$, many interesting geometric structures can arise.  For instance, if $M$ is a closed surface with  genus at least 2, a hyerbolic
structure corresponds to $(\text{PSL}(2,\mathbb R),\mathbb H^2)$, a real projective structure  to $(\PGL(3,\mathbb R), \RP^2)$, a complex projective structure to $(\text{PSL}(2,\mathbb C), \mathbb{CP}^1)$. In this paper, we are concerned with an affine structure $(\text{Aff}(n,\br), \br^n)$ where
$\text{Aff}(n,\br)$ is the affine group $\GL(n,\br)\ltimes \br^n$.

Given a geometric structure, there exist a developing map $D:\widetilde M \ra X$ and a holonomy homomorphism $\rho:\pi_1(M)\ra G$ so that $D$ is $\rho$-equivariant. When the developed image $D(\widetilde M)=X$, it is said that $M$ has a complete $(G,X)$ structure.
 Depending on the topology of $M$, only specific complete geometric structures are allowed. For example, if $M$ is a torus, $M$ cannot have a complete hyperbolic structure. This fact can be seen also via Gauss-Bonnet theorem.

Sometimes we are concerned with a complete affine structure with a specific holonomy group. It is said that an affine manifold has  {\bf parallel volume} if the linear part of the holonomy group lies in $\SL(n,\br)$.

In 1962, Markus conjectured that a compact affine manifold $M$ with parallel volume is complete, i.e., $M=\br^n/\Gamma$ where $\Gamma$ is a discrete subgroup of $\Aff(n,\br)$, see \cite{Ma}.

When the developed image $D(\widetilde M)$ is a convex domain $\Omega$ in $\br^n$, it is said that $M$ has a convex affine structure.
 Markus conjecture says that such an affine manifold cannot have a parallel volume if $\Omega \neq \br^{n}$. In this paper, we prove this 
under the condition that the limit set $\Lambda_{\Aut^0(\Omega)}$, the set of all limit points on the boundary under the action of the identity component of $\Aut(\Omega)$, is nonempty. Here $\Aut(\Omega)=\{g\in
\GL(n,\br)\ltimes \br^n|\ g(\Omega)=\Omega\}$, is the group of affine automorphisms of
$\Omega$.
\begin{theorem}\label{Markus}
Let $M$ be a compact convex affine manifold whose developing image $\Omega$ is a proper subset of $\br^n$. Then $M$
cannot have parallel volume if $\Lambda_{\Aut^0(\Omega)}$ is
nonempty.
\end{theorem}
Furthermore we will see  in section \S \ref{dim5} that the Markus conjecture is  true for convex case when its dimension is $\leq 5$ through the classification of their developing images.
\begin{theorem}\label{main2}
The Markus conjecture is true for convex affine manifolds of dimension $\leq 5$.
\end{theorem}

Though there are several partial results concerning this conjecture \cite{C,Fr,FG,GH1,Jo}, yet it is far from being completely resolved.  Goldman and Hirsch \cite{GH2} showed that the affine holonomy $\Gamma$ of a compact affine $n$-manifold with parallel volume preserves no proper (semi)-algebraic subset of $\br^n$, and in fact the algebraic hull
$A(\Gamma)$ of $\Gamma$ acts transitively on $\br^n$. We will use this fact in the proof of our main theorem.
We also proved the conjecture under the some suitable assumption on projective automorphism of the domain \cite{JK} to which the current paper is a sequel.
We hope to resolve the conjecture  without the assumption about the limit set in due time. See Proposition \ref{noncpt}.

\section{preliminaries}
In this section, we review the basic concepts and properties for convex domains in projective space. To begin with, let us be precise about our terminology. When we speak of a simplex in this paper, we mean the domain which consists of all the points in the interior of the simplex. The same is true for a polyhedron, an ellipsoid,  a paraboloid, an elliptic cone etc.
\begin{Def}
 A subgroup $G$ of $\text{Aff}(n,\br)$ is said to be $irreducible$
if $G$ preserves no proper affine subspace of $\br^n$.
\end{Def}

Sometimes we  look at domains and their automorphisms in the
projective space $\RP^n=(\br^{n+1}\setminus \{0\})/\br^*$ and in
$\PGL(n+1,\br)$, where $\mathbb R^*=\mathbb R \setminus \{0\}$.
Naturally $\PGL(n+1,\br)$ acts on $\RP^n$ as projective
transformations. Denote the projectivization of $n$ by $n$
matrices by $\PM(n,\br)$. Recall that a domain in $\br^n$ can be
viewed as a domain in $\RP^n$ whose automorphism group preserves
the set of points at infinity, $\mathbb {RP}^{n-1}_{\infty }$, by
identifying $\mathbb R^n$ with the affine space given by
$x_{n+1}=1$ in $\mathbb R^{n+1}$ so that $\mathbb {RP}^n$ becomes
a compactification of $\mathbb R^n$.

Since $\mbox{PM}(n+1,\mathbb R)$ is a compactification of
$\PGL(n+1,\br)$, any infinite sequence of non-singular projective
transformations contains a convergent subsequence in
$\mbox{PM}(n+1,\mathbb R)$. Note that the limit projective
transformation $g$  of a sequence of non-singular projective
transformations ${g_i}$ might be singular. We will denote the projectivization of the
kernel and range of $g$ by $K(g)$ and $R(g)$. Then $g$ maps $\RP^n
\setminus K(g)$ onto $R(g)$, and the $g_i$-images of any compact set in
$\RP^n \setminus K(g)$ converges uniformly to the image under the limit transformation $g$
of ${g_i}$ (see \cite {Ben}). This implies that $g_i(p_i)$ converges to $g(p)$ if $p_i$ converges to $p\notin K(g)$.

\begin{Def} Let $\Omega$ be a convex projective domain and $\overline \Omega$ be the closure of $\Omega$ in the projective space.
\begin{enumerate}
\item[\rm (i)] { $\Omega$ is called \emph{properly convex} if it
does not contain any complete line.}
\item[\rm (ii)] {A \emph{face} of $\Omega$ is an equivalence class with respect to
the equivalence relation given as follows.
$x\sim y$ if either $x=y$ or $\overline \Omega$ has an open line
segment containing both $x$ and $y$. Then a face is a relatively
open convex subset of $\overline \Omega$ and $\overline \Omega$ is a
disjoint union of faces.}
\item[\rm (iii)] {A \emph{support} of a face $F$ is the subspace generated by $F$. We will denote the support of $F$ by $\langle F \rangle$.}
\item[\rm (iv)] {The \emph{dimension} of a face $F$ is the dimension of the
support.}
\item[\rm (v)] {A \emph{supporting subspace} $V$ of $\Omega$ is a subspace containing
every support of a face intersecting $V$.}
\item[\rm (vi)]{Zero dimensional face is called  an \emph{extreme point}.}
\end{enumerate}
\end{Def}
 We say that $E$ is a \emph{closed face} of $\Omega$ if $E=\overline{F},$ for a face $F$ of  $\Omega$.
\begin{Def}\label{def-Ben}
\mbox{}
Let $\Omega$ be a properly convex domain in $\rp{n}$.
\begin{enumerate}
\item [\rm (i)] Let $\Omega_1$ and
$\Omega_2$ be convex domains in $\langle \Omega_1 \rangle$ and
$\langle \Omega_2 \rangle$ respectively. $\Omega$ is called a
\emph{convex sum} of $\Omega_1$ and $\Omega_2$, which will be
denoted by $\Omega = \Omega_1 \dot{+} \Omega_2$, if $\langle \Omega_1
\rangle \cap \langle \Omega_2 \rangle =\emptyset$ and $\Omega$ is
the union of all open line segments joining points in $\Omega_1$
to points in $\Omega_2$. We say that $\Omega$ is \emph{decomposable} if $\Omega$ has such a decomposition. Otherwise, $\Omega$ is called \emph{indecomposable}.
\item [\rm (ii)]A $k$-dimensional face $F$ of an $n$-dimensional convex domain $\Omega$ is
called \emph{conic} if there exist $n-k$ supporting hyperplanes
$H_1, H_2, \dots,H_{n-k}$ such that
\begin{equation}\label{conic-hyperplanes}
H_1 \gneq H_1 \cap H_2 \gneq \dots \gneq H_1 \cap \dots \cap
H_{n-k} = \langle F \rangle. 
\end{equation}
Especially, a codimension one face is conic.
\item [\rm (iii)] We say that $\Omega$ has an \emph{osculating  ellipsoid} at
$p \in
\partial \Omega$ if there exist a suitable affine chart and a basis
such that the local boundary equation on some neighborhood of
$p=(0,\dots,0)$ is expressed by $x_n = f(x_1,\dots, x_{n-1})$ and
\begin{equation*} \lim_{(x_{1},\dots,x_{n-1})\to 0}
\frac{f(x_{1},\dots,x_{n-1})}{{x_1}^2+ \dots +{x_{n-1}}^2}=1.
\end{equation*}
\end{enumerate}
\end{Def}
\begin{remark}\label{Ben-oscullating}
\begin{enumerate}
    \item [\rm (i)]
If $\partial \Omega$ is twice differentiable on a neighborhood of a strictly convex boundary point $p$ of $\Omega$, then $\Omega$ has an osculating  ellipsoid at
$p$. See \cite{Jo}.
	\item [\rm (ii)] We will see in section \ref{Ben-Result} that when $\Omega$ is quasi-homogeneous, the property that a face $F$ is conic is equivalent to the property that $F$ is a convex summand of $\Omega$. See  Theorem \ref{section-benz}.
\end{enumerate}
\end{remark}

Let  $\Aut_{\text{proj}}(\Omega)$ be the set of all projective transformations which preserves a  domain $\Omega \subset \RP^n$. Since there is a natural surjection $\pi : \text{SL}^{\pm}(n+1,\br) \rightarrow \PGL(n+1,\br)$  with its kernel $\{\text{Id}\}$ or $\{\text{Id}, -\text{Id}\}$, where  $\text{SL}^{\pm}(n+1,\mathbb R)$ is the group of linear transformations of $\mathbb R^{n+1}$ with determinant $1$ or $-1$,  we will regard $\Aut_{\text{proj}}(\Omega)$ as a subgroup of the group $\text{SL}^{\pm}(n+1,\mathbb R)$ from now on. Note that if $\Omega$ is a  domain in $\mathbb R^n$ then we see that $\Aut(\Omega)$ is a subgroup of $\Aut_{\text{proj}}(\Omega)$ and
$\Aut(\Omega)=\Aut_{\text{proj}}(\Omega)\cap \text{Aff}(n,\br).$

It is well-known that  if $\Omega$ is properly convex, then there is a
complete continuous metric which is invariant under the action of
$\Aut_{\text{proj}}(\Omega)$. This metric is called the Hilbert metric and
defined as follows: Any properly convex projective domain is projectively equivalent to a bounded convex domain in an affine space, we may assume that $\Omega$ is a bounded convex domain in $\mathbb R^n$.
\begin{Def}
Let $\Omega$ be a bounded convex domain in $\mathbb R^n$.
For any two different points $p_1, p_2 \in \Omega$, we define
$\mbox{d}_{\Omega}(p_1, p_2)$ to be the logarithm of the absolute value
of the cross ratio of $(s_1, s_2, p_1, p_2)$, 
 where $s_1$ and
$s_2$ are the points in which the line $\overleftrightarrow {p_1p_2}$
intersects $\partial \Omega$. That is, if the four points  are in a sequence  $s_1, p_1, p_2, s_2$ then 
$$\mbox{d}_{\Omega}(p_1, p_2)=\ln \frac{\mbox{d}_{\mathbb R^n}(s_1, p_2)\mbox{d}_{\mathbb R^n}(p_1, s_2)}{\mbox{d}_{\mathbb R^n}(s_1, p_1)\mbox{d}_{\mathbb R^n}(p_2, s_2)}.$$
For $p_1 = p_2$, we define
$\mbox{d}_{\Omega}(p_1, p_2) =0$.
\end{Def}

So when $\Omega \subset \RP^n$ is properly convex, any element $\gamma$ of $\Aut_{\text{proj}}(\Omega)$ is an isometry of the Hilbert metric. We say $\gamma$ is $elliptic$ if it fixes a point in $\Omega$. If $\gamma$ acts freely on $\Omega$ it is called $parabolic$ if every eigenvalue has modulus 1 and $hyperbolic$
 otherwise. The $translation$ $length$ of $\gamma$,
$$t(\gamma)=\text{inf}_{x\in\Omega}\quad \mbox{d}_{\Omega}(x,\gamma(x)),$$
 is equal to  the logarithm of the absolute value of the ratio of the eigenvalues of $\gamma$ of maximum modulus $\lambda$ and minimum modulus $\mu$, i.e., $$t(\gamma)=\text{ln}|\lambda/\mu|.$$

Let $p$ be a boundary point of $\Omega$ and $H$ a supporting hyperplane to $\Omega$ at $p$. Define  $S_0$ to be the subset of boundary of $\Omega$ obtained by deleting $p$ and all the line segments in $\partial\Omega$ with end points at $p$.  A $generalized$ $horosphere$ $centered$ $on$ $(H,p)$ is the image of $S_0$ under the action of $G(H, p)$, the set of all the projective transformations  which translate towards $p$ and fix  $H$ pointwise, which is equal to the set of all the affine translations towards $p$ in the affine space $\RP^n\setminus H$.
 We simply call a generalized horosphere a $horosphere$ in $\Omega$.
Then $\Omega$ is foliated by horospheres and this foliation is preserved by an automorphism $A\in \Aut_{\text{proj}}(\Omega)$   if $A$ preserves $p$ and $H$. In particular, each horosphere centered on $(H,p)$ is preserved if $A$ is a parabolic isometry fixing $(H,p)$. (See \cite{CLT} or \cite{Marquis} 
 for a reference.)

The following  figures show two kinds of horospheres of the triangle under the action of a hyperbolic isometry $A \in \PGL(3,\br)$ with eigenvalues $2, 2, 1/4$. Figure \ref{HS} shows the horospheres in $\RP^2$ and
Figure \ref{affineHS} shows them in the affine space which is the complement of the supporting hyperplane  $ H$.
\begin{figure}[hbt]
\begin{center}
\scalebox{1.0}{\includegraphics{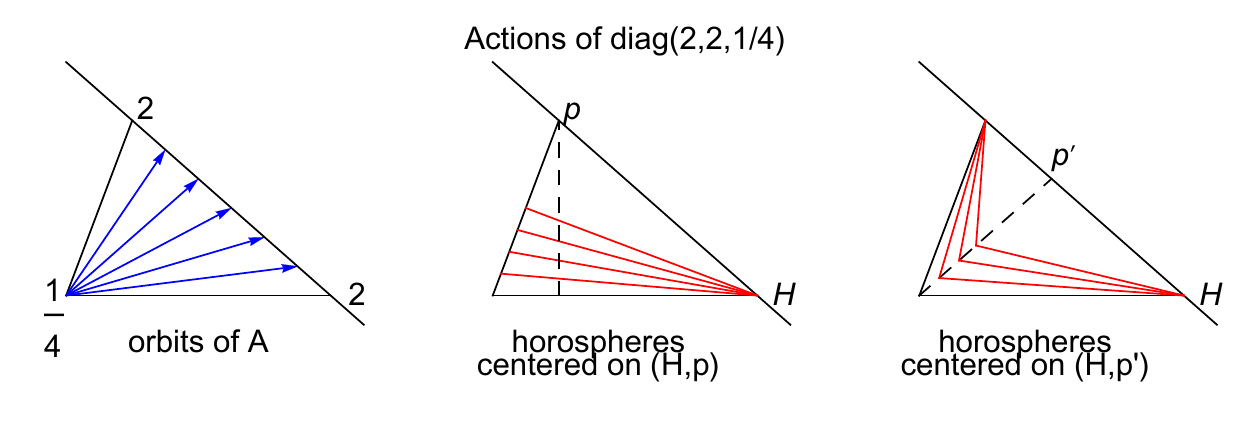}}
\end{center}
\caption{Horospheres of a triangle in $\RP^2$ }\label{HS}
\end{figure}

\begin{figure}[hbt]
\begin{center}
\scalebox{1.0}{\includegraphics{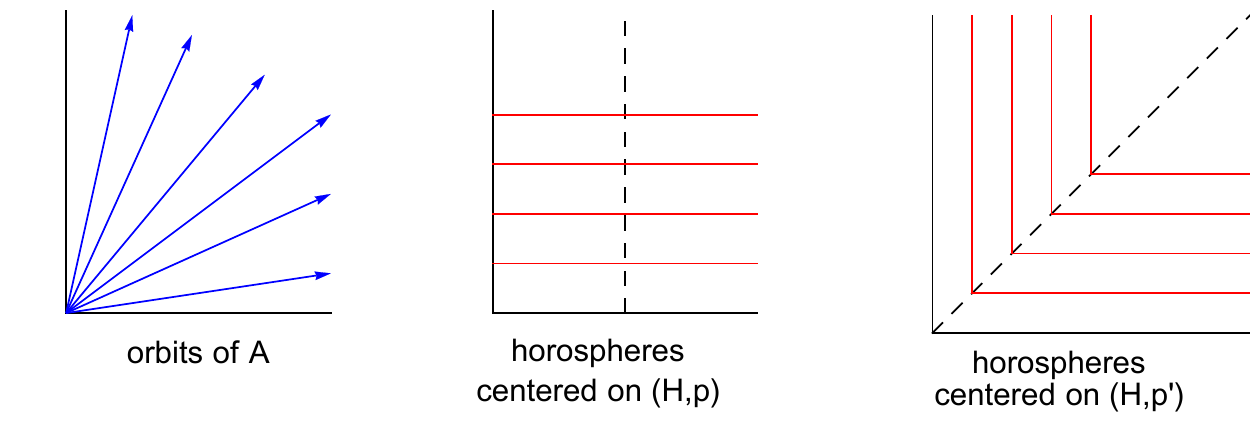}}
\end{center}
\caption{Horospheres of a triangle in the affine space $\RP^2\setminus H$  }\label{affineHS}
\end{figure}

The next theorem  will be used later.
\begin{theorem}[Jo, \cite{Jo4}]\label{hess-noncompt} Let $\Omega$ be a domain in $\mathbb {RP}^n$.
 Then $\Omega$ is an ellipsoid if and only if $\Omega$ has a locally strictly convex point $p$ in the boundary (that is, there exists a connected open neighborhood $U$ of $p$ such that
$U\cap \Omega$ is a strictly convex domain) such that
\begin{enumerate}
\item[\rm(i)] $\partial \Omega$ is $C^2$ near $p$,
\item[\rm(ii)] the Hessian is non-degenerate at $p$,
\item[\rm(iii)] $\Aut_{proj}(\Omega)x$ accumulates at $p$ for some $x\in \Omega$.
\end{enumerate}
\end{theorem}


\section{Quasi-homogeneous domains}
A domain $\Omega$ is called $ homogeneous$ if $\Aut(\Omega)$ acts
transitively on $\Omega$ and $quasi$-$homogeneous$ if there exists a compact set $K\subset \Omega$ and
$G\subset \Aut(\Omega)$ so that $GK=\Omega$. In this case we also say that $G$ acts on $\Omega$ $syndetically$.
 $\Omega$ is called  $ divisible$ if there
exists a discrete subgroup $\Gamma\subset \Aut(\Omega)$ so that
$\Omega/\Gamma$ is a compact manifold.

Note that both homogeneous and divisible domains are quasi-homogeneous and any compact convex affine $n$-manifold $M$ has a divisible domain $\Omega$ in  $\br^n$ and a discrete subgroup $\Gamma$ of  $\text{Aff}(n,\br)$ acting on $\Omega$ such that $M= \Omega/\Gamma$. Furthermore the following surprising theorem is  well-known.
\begin{theorem}[Vey, \cite{V3}]\label{VeyThm}
A divisible properly convex affine domain is a cone.
\end{theorem}
We list out some useful facts.
\begin{Prop}[Vey, \cite{V3}]\label{Vey}
Let $\Omega$ be a quasi-homogeneous properly convex affine domain. Then
\begin{enumerate}
\item [\rm (i)] For any $x\in\Omega$ and extreme point $\xi$, there exist $g_i\in G$ such that $g_ix\rightarrow \xi$.
\item [\rm (ii)]
$\Omega=CH(Gx) $ for any $x\in\Omega,$
\item [\rm (iii)] If $L$ is a $G$-invariant proper affine subspace of $\br^n$, then
$$L\cap\Omega=\emptyset \text{ and } L\cap\partial\Omega\neq\emptyset.$$
\end{enumerate}
\end{Prop}
Here $CH(Gx) $ means the convex hull of $Gx$.

\begin{Prop}[Jo, \cite{Jo, Jo5}]\label{bdd-face}
A quasi-homogeneous convex affine domain cannot have any bounded face with non-zero dimension.
\end{Prop}

For a strictly convex quasi-homogeneous projective domain $\Omega\subset \RP^n$,
the following is proved in \cite{Jo}.
\begin{Prop}\label{strictlyCV} Let $\Omega$ be a strictly convex quasi-homogeneous domain in $\RP^n$. Then 
\begin{enumerate}
\item [\rm (i)]$\partial \Omega$ is at least
$C^1$,
\item [\rm (ii)]$\Omega$ is an ellipsoid if and only if $\partial \Omega$ is
twice differentiable,
\item [\rm (iii)]if $\partial \Omega$ is $C^{\alpha}$ on an open subset of $\partial \Omega$, then
$\partial \Omega$ is $C^{\alpha}$ everywhere,
\item [\rm (iv)] $\Omega$ is an ellipsoid  if and only if a boundary point of $\Omega$ is fixed by a subgroup $G$ of $\Aut_{proj}(\Omega)$ acting on $\Omega$ syndetically.
\end{enumerate}
\end{Prop}

Every quasi-homogeneous convex affine domain contains a cone invariant under the action of linear
parts of their automorphism groups, which is called an asymptotic
cone. This terminology was originally introduced by Vey in
\cite{V3}.
\begin{Def}
 Let $\Omega$ be a convex domain in $\mathbb
R^n$. The {\it asymptotic cone} of $\Omega$ is defined as follows:
\begin{equation*}
\text{AC} (\Omega)=\{u\in \mathbb R^n\,|\, x+tu\in \Omega, \mbox{
for all } x\in \Omega, t\geq 0\}.
\end{equation*}
\end{Def}
By the convexity of $\Omega$, for any $x_0 \in \Omega$,
\begin{equation*}
\text{AC} (\Omega)=\text{AC}_{x_0} (\Omega): =\{u\in \mathbb
R^n\,|\, x_{0}+ tu\in \Omega, \mbox{ for all }t\geq 0\}.
\end{equation*}
Note that $\mbox{AC} (\Omega )$ is a properly convex closed cone
in $\mathbb R^n$ if $\Omega $ is properly convex. If we denote
the interior of $\mbox{AC} (\Omega )$ relative to its affine hull
by $\mbox{AC}^\circ (\Omega )$, then it is proved in \cite{Jo1} that $\mbox{AC}^\circ (\Omega )$ is a (quasi)-homogeneous domain if $\Omega$ is (quasi)-homogeneous. 

Even though the asymptotic cone $\text{AC}(\Omega )$ of a properly convex affine domain $\Omega$ is possibly empty, it is nonempty if $\Omega$ is quasi-homogeneous because there is no bounded quasi-homogeneous convex domain (see Proposition \ref{bdd-face}).
Vey proved in \cite{V3} that a quasi-homogeneous properly convex affine domain is itself a cone if the dimension of its asymptotic cone is equal to the dimension of $\Omega$. More generally, we get the following.

\begin{theorem}[Jo, \cite{Jo}]\label{thm-folliation}
Let $\Omega$ be a properly convex quasi-homogeneous affine domain
in $\mathbb R^n$. Then

\begin{enumerate}
\item [\rm (i)]
$\Omega $ admits a parallel foliation by
cosets of the asymptotic cone $\emph{AC}^\circ(\Omega )$ of $\Omega$, we call this asymptotic foliation of $\Omega$,
\item [\rm (ii)]the set of all the asymptotic cone points of $\Omega$ is equal to the set of all the extreme points of $\Omega$, where an asymptotic cone point is a boundary point $\xi$ such that $\xi+\emph{AC}^\circ(\Omega )$ is a leaf of the above foliation.
\end{enumerate}
\end{theorem}
Note that if we consider $\Omega$ as a projective domain in $\mathbb {RP}^n$, then the set of all  extreme points of $\Omega$ is the union of the set of all  asymptotic cone points and  {infinite} extreme points of $\Omega$.

\begin{theorem}[Jo, \cite{Jo2}]\label{transitive}
Let $\Omega $ be a properly convex affine domain in
     $\mathbb R^n$ and $G$ be a closed subgroup of  $\Aut(\Omega)$ acting syndetically on $\Omega$. Then $G$
      acts transitively on the set $S(\Omega)$ of all  asymptotic cone points of $\Omega$.
\end{theorem}

\begin{lemma}\label{extreme}
Let $\Omega$ be a properly convex quasi-homogeneous affine
domain in $\br^n$. Then any boundary point which is a limit point of a sequence of extreme points is again an extreme point. Especially, for any $\gamma \in \overline{\Aut(\Omega)}$ and an extreme point $\xi$, $\gamma(\xi)$ is an extreme point if it is in $\br^n$ (it is an infinite extreme point if it is in $\partial_{\infty}\Omega$).
\end{lemma}
\begin{proof}
Since every extreme point is an asymptotic cone point and an asymptotic foliation is preserved by $\Aut(\Omega)$, any limit point of a sequence of extreme points must be an extreme point.

\end{proof}

The following corollary will be used later in this paper.
\begin{co}\label{AC-Face}
Let $\Omega$ be a quasi-homogeneous properly convex affine domain in $\mathbb R^n$ and $S$ be the set of extreme points of $\Omega$. Then  the following holds.
\begin{enumerate}
\item [\rm (i)] For any extreme point $\xi$,
$$(\xi+\emph{AC}(\Omega ))\cap \Omega=\xi+\emph{AC}^\circ (\Omega ),$$
that is, $\xi+\emph{AC}^\circ (\Omega )$ cannot be a face of $\Omega$ unless $\Omega$ itself is a cone.
\item [\rm (ii)] If $\Omega=F_1\dot{+}F_2$, then $\bar{F_i} \cap \partial_{\infty}\Omega\neq\emptyset$ for $i=1,2$, and $S$ is contained entirely in one of $F_1$ and $F_2$. So either $S\subset F_1$ and $ F_2 \subset \partial_{\infty}\Omega$ holds or $S\subset F_2$ and $F_1 \subset \partial_{\infty}\Omega$ holds.
\item [\rm (iii)]  If $\overline\Omega=\overline{CH(S)}$, then $\Omega$ cannot have any conic face and $\Aut(\Omega)$ is irreducible.
\end{enumerate}
\end{co}
\begin{proof}
\begin{enumerate}
\item [\rm (i)] For any point $x\in\Omega$, the cone point $\xi_x$ of the leaf of asymptotic foliation of $\Omega$, which contains $x$, is an extreme point by Theorem \ref{thm-folliation} and thus $\xi_x+\emph{AC}^\circ (\Omega )$ is included in $\Omega$. 
Since  $\Aut(\Omega)$ acts transitively on $S$, $\xi+\emph{AC}^\circ (\Omega )$ should consist of interior points of $\Omega$ for all $\xi \in S$.
\item [\rm (ii)] Both $\bar{F_1}$ and $\bar{F_2}$ have an infinte boundary point since $\Omega$ cannot have any {non-zero dimensional} bounded face by Proposition \ref{bdd-face}. By Theorem \ref{thm-folliation} and Lemma \ref{extreme}, $S$ is connected and thus either $S\subset \bar{F_1}$ or $S\subset \bar{F_2}$ holds.
\item [\rm (iii)]  Suppose $\Omega$ has a conic face $F_1$. Then there is a face $F_2$ of $\Omega$ such that $ F_2 \subset \partial_{\infty}\Omega$ and $\Omega=F_1\dot{+}F_2$ (see (ii) of Theorem \ref{section-benz} in the next section). This implies that $S\subset F_1$ and $\Omega$ is affinely equivalent to 
$F_1 \times C(F_2)$, where $C(F_2)$ is a cone over $F_2$. But this is a contradiction because 
 $\overline{CH(S)}\subset \overline {F_1} \neq \overline\Omega$.
By Theorem \ref{transitive}, $\Aut(\Omega)$ cannot have any invariant proper subspace.
\end{enumerate}
\end{proof}

\section{Benz\'{e}cri's result}\label{Ben-Result}

The structure of quasi-homogeneous domains have been studied a lot in convex case since Benz\'{e}cri. Here are some results about (quasi-homogeneous) convex domains which are needed later in this paper.
\begin{theorem}[Benz\'{e}cri, \cite{Ben}]\label{prop-benz}
Let $\Omega$ be a properly convex domain in $\rp{n}$.
\begin{enumerate}
\item [\rm(i)]
If $\Omega$ has an osculating ellipsoid
$Q$, then there exists a sequence $\{g_{n}\}\subset  \PGL(n+1,\br)$
such that $g_{n}\Omega $ converges to $Q$.
\item [\rm(ii)]
Let $F$ be a conic face of $\Omega$. Then there exist a projective subspace $L$ of $\rp{n}$ and
projective automorphisms $\{h_i\}$ of $\rp{n}$
such that $\{h_i\Omega \}$ converges to $F\dot{+}B$ for some properly
convex domain $B$ in $L$.
\item[\rm(iii)]
If $\Omega = {\Omega}_1 \dot{+}{\Omega}_2$, then $\Omega$ is
quasi-homogeneous (respectively, homogeneous) if and only if ${\Omega}_i$
is quasi-homogeneous (respectively, homogeneous) for each i.
\end{enumerate}
\end{theorem}
\begin{theorem}[Benz\'{e}cri, \cite{Ben}]\label{section-benz}
Let $\Omega$ be a quasi-homogeneous properly convex domain in $\rp{n}$. 
\begin{enumerate}
\item [\rm(i)]
If $\Omega$ has an osculating ellipsoid, then $\Omega $ is projectively equivalent to a ball. 
\item [\rm(ii)]
If $\Omega$ has a conic face $F$ of $\Omega$, then there exists another conic face $B$ of $\Omega$  such that $\Omega=F\dot{+}B$.
\end{enumerate}
\end{theorem}
From Benz\'{e}cri's result, we get the following lemmas, which are proved in \cite{Jo}.

\begin{lemma}\label{singular} Let $\Omega$ be a properly convex
domain in $\RP^n$. Suppose a sequence $g_i\in \emph\Aut_{\emph{proj}}(\Omega)$
converges to a singular projective transformation $g\in \PM(n+1,\br)$. Then
$K(g)$ and $R(g)$ do not meet $\Omega$ and the following holds:

\begin{enumerate}
\item [\rm (i)]{$K(g)$ is a supporting subspace of $\Omega$}.
\item [\rm (ii)]{$R(g)$ is a support of a proper face $F$ of $\Omega$}.
\item [\rm (iii)]{If $x_0 \in \Omega$ and $\lim_{i\rightarrow\infty} g_i x_0 = \xi \in \partial{\Omega}$, then $R(g)$ is the support of
the face $F$ containing $\xi$, { that is, $R(g)=\langle F \rangle$.}
\item [\rm (iv)]$g(\Omega)=F$, that is, for any point $\eta \in F$ there
exists $x \in \Omega$ so that  $\lim_{i\rightarrow\infty} g_i x=\eta$. }
\end{enumerate}
\end{lemma}
{Note that $R(g)$ is equal to $\langle R(g) \cap \overline{\Omega} \rangle$, but   $K(g)$ might properly contain  $\langle K(g) \cap \overline{\Omega} \rangle$.}

\begin{lemma}[Jo, \cite{Jo}]\label{kernel}
Let $\{f_{i}\}$ be a sequence in $\text{Aff}(n,\br)$. Suppose that $f_i$ converges to $f\in  \mbox{PM}(n+1,\mathbb
R)$ with $R(f)\cap \mathbb R^n \neq \emptyset$. Then $K(f)\cap
\mathbb R^n=\emptyset$.
\end{lemma}

\begin{lemma}[Jo, \cite{Jo}] \label{lem-saillant}
Let $\Omega$ be a quasi-homogeneous properly convex domain in
$\rp{n}$and $G$ a subgroup of $\mbox{Aut} ( \Omega ) $ acting on
$\Omega$ syndetically. Then for each point $p\in \partial \Omega$,
there exists a sequence $\{g_i\}\subset G$ and $x\in \Omega$ such
that $g_{i}(x)$ converges to $p$.
\end{lemma}
\begin{lemma}\label{supker}Let $\Omega$ be a properly convex projective domain and $\{\phi _n\} \subset \Aut
	(\Omega)$ a sequence converging to $\phi \in \mbox{PM}(n+1,\mathbb
	R)$
	with kernel $K(\phi)$ and range $R(\phi)$. Suppose
	$K(\phi)=\langle K(\phi)\cap \partial \Omega \rangle$, i.e.,
	$K(\phi)$ is a support of some face of $\Omega$. Then $K(\phi)\cap
	\partial \Omega $ is a {closed} conic face of $\Omega$.
\end{lemma}

\begin{proof}
	Let  $k$ be the dimension of $K(\phi)$ and $F$ be the $k$-dimensional face of $\Omega$ such that $\overline{F}=K(\phi)\cap
	\partial \Omega $. We can choose a
	complementary subspace $L$ of dimension $(n-k-1)$, that is,
	the space generated by $K(\phi)$ and $L$ is the whole $\rp{n}$. Let $\rho$ be a projection from $\rp{n}$ to
	$L$. Then $\phi$ can be considered as a projective transformation
	from $L=\rho(\rp{n})$ to $R(\phi)$, since $\phi$ maps the
	projective space generated by $K(\phi)$ and $y \in L$ to one point
	$\phi(y)$ in $R(\phi)$. The fact that $K(\phi)$ does not intersect
	$\Omega$ implies that $\phi(x)=\lim_{n \to \infty} \phi _n(x)$
	must be contained $R(\phi)\cap\partial\Omega$ for all $x\in
	\Omega$. Since $\Omega$ is properly convex,
	$R(\phi)\cap\partial\Omega$ is also properly convex and so
	$\rho(\Omega)$ is properly convex in $L$. Therefore $\rho(\Omega)$
	is bounded by $(n-k)$ number of $(n-k-2)$-planes $\{
	{H_L}^1,\dots,{H_L}^{n-k} \}$ of $L$ which bound a
	$(n-k-1)$-simplex. If we let $H_i$ be the $(n-1)$-plane generated by $K(\phi)$ and ${H_L}^i$ for each $i$, then $\{H_i\}$ are hyperplanes of
	$\rp{n}$ which satisfies (\ref{conic-hyperplanes}). Therefore 
	$F$ is a conic face of
	$\Omega$.
	
\end{proof}


\section{Limit set}

\begin{Def}Let $\Omega\neq \br^n
$ be a domain in $\br^n$ and $G < \Aut(\Omega)$.
A limit set $\Lambda_G \subset \br^n $ of $G$ is
$$\cup_{x\in \Omega}(\overline{Gx}\cap \partial \Omega)$$
\end{Def}

\begin{lemma}\label{inv}Let $\Omega$ be a properly convex
 affine domain in $\br^n$. Suppose
$H $ is a normal subgroup of $\Aut(\Omega)< \text{Aff}(n,\br)$. Then
$\Aut(\Omega)$ leaves $\Lambda_H$ invariant.
\end{lemma}
\begin{proof}
Suppose $\xi$ is a point of $\Lambda_H$ and $g$ is an element of
$\Aut(\Omega)$. Then there exists a sequence $\{h_i\}$ of $H$ and
a point $x_0$ of $\Omega$ such that $h_i(x_0)$ converges to $\xi$.
If $y_0=g(x_0)$, then $gh_ig^{-1}(y_0)$ converges to $g(\xi)$
since
$$\lim_{i\rightarrow \infty} gh_ig^{-1}(y_0)=ghg^{-1}(y_0)=gh(x_0)=g(\xi).$$ This implies that $g(\xi)$ is
contained in $\Lambda_H$ and thus we can conclude that $\Aut(\Omega)$ leaves $\Lambda_H$ invariant.
\end{proof}
{Note that  by Lemma \ref{singular} the limit set $\Lambda_H$ is the disjoint union of the boundary face of $\Omega$ whose support is the range space of some $h\in \overline{H}-H$.}

\begin{lemma}\label{closed}
Let $\Omega$ be a properly convex quasi-homogeneous affine
domain in $\br^n$. Suppose that $G=\Aut(\Omega)$ is irreducible and the limit set $\Lambda_{G^0}$  of the identity component of $G$ is nonempty.
Then $\overline{CH(\Lambda_{G^0})}=CH(\overline{\Lambda_{G^0}})=\overline\Omega$ and $\overline{\Lambda_{G^0}}$  contains all the extreme points.
\end{lemma}
\begin{proof}
Since $G$ is irreducible and preserves $\Lambda_{G^0}\neq \emptyset$,
$CH(\Lambda_{G^0})\cap \Omega\neq\emptyset$.
Now we get from Proposition \ref{Vey}
$$\overline{CH(\Lambda_{G^0})}=CH(\overline{\Lambda_{G^0}})=\overline\Omega$$
 by considering $G$-orbit of a point $x \in CH({\Lambda_{G^0}})\cap \Omega$.
Hence every extreme point is  in $\overline{\Lambda_{G^0}}$, that is, $S\subset \overline{\Lambda_{G^0}}.$
\end{proof}

We will see in Proposition \ref{homo} that under the assumptions in Lemma \ref{closed}, the domain is homogeneous. Hence  the limit set $\Lambda_{G^0}$  is $\partial\Omega$ containing all the extreme points.



\section{Proof of Theorem \ref{Markus}}

To prove theorem \ref{Markus} we first show the following two propositions.

\begin{Prop}\label{homo}
Let $\Omega$ be a properly convex
quasi-homogeneous affine domain in  $\br^n$ which has an irreducible $G=\Aut(\Omega)$ and a nonempty $\Lambda_{G^0}$.
Then $\Omega$ is a homogeneous affine domain.
\end{Prop}
\begin{proof}
Consider an orbit $G^0x$ of a point $x\in\Omega$. Since $G^0x$ accumulates at every face of $\Omega$ which is in 
$\Lambda_{G^0}$, we can find a one-parameter subgroup $g_E(t)$ for each face $E \subset \Lambda_{G^0}$ which satisfies the following: $g_E(t)$ is parabolic or hyperbolic, $E$ is fixed by $g_E(t)$,
$g_E(t)$ preserves a supporting hyperplane $H^E$ which contains $E$ for all $t$,
and
$g_E(t)x$ converges to a point of $E$ as $t$ goes to $\infty.$

We denote the set of fixed points of  $g_E(t)$ by $F_E$ for each face $E \subset \Lambda_{G^0}$,  {and especially if $g_E(t)$ is hyperbolic, we denote $F_E^+$ ($F_E^-$, respectively) the set of attracting fixed points (repelling fixed points, respectively) and $F_E^0$ is the set of the remaining fixed points.}    
We call the horospheres centered on $(H^E,p), p\in E$ as  $\{g_E(t)\}$-horospheres. Note that these horospheres make a foliation of $\Omega$.

Such a one-parameter subgroup $g_E(t)$ can be constructed as follows.  Choose $g_n\in G^0$ such that $g_nx\ra p\in E$.  Let $g_n(t)=\exp(t \eta_n),\ \eta_n \in \mathfrak{g}^0$, be a one-parameter subgroup in $G^0$ so that $g_n(t_n)=g_n$ for $t_n\ra\infty$. Here we normalize $\eta_n$ so that its norm is 1 with respect to a Killing form on $\text{SL}(n+1,\mathbb R)$. Since the set of directions $\mathfrak{g}^0\cap {\bold S}$, where ${\bold S}$ is a unit sphere in $\mathfrak{sl}(n+1,\mathbb R)$, is compact, we can pass to a subsequence so that  $\eta_n\ra\eta$. Set $g_E(t)=\exp(t \eta)$. Then the one-parameter subgroups $g_n(t)$ converges to $g_E(t)$ and
\begin{equation*}
p=\lim_{n\ra\infty} g_nx=\lim_{n\ra\infty} g_n(t_n)x=\lim_{n\ra\infty} \exp(t_n\eta_n)x=\lim_{t\ra\infty} \exp(t \eta) x=\lim_{t\ra\infty}g_E(t)x.
\end{equation*}
Obviously $g_E(t)$ is not elliptic for each $t$  and thus it preserves a supporting hyperplane $H_t^E$ to $\Omega$ at $p$ by Lemma 2.3 of \cite{CLT}. Since one parameter subgroup is abelian, it preserves a common supporting hyperplane $H^E$.

By Lemma \ref{closed}, we can choose two distinct faces $E_1, E_2$ in $\Lambda_{G^0}$ any of which is not contained in the closure of the other. Let $H_{E_1E_2}$ be the subgroup of $G^0$ which is generated by $g_{E_1}(t)$ and $g_{E_2}(s)$.
 If $\Omega$ is a $2$-dimensional domain, then we can prove $H_{E_1E_2}x=\Omega$ as follows: If
 both $g_{E_1}(t)$ and $g_{E_2}(s)$ are hyperbolic and an arbitrary pair of two orbits  $\{ g_{E_1}(t)x\,|\, t\in\mathbb R\}$ and $\{g_{E_2}(s)y\,|\, s\in\mathbb R\}$  intersect transversely, then $H_{E_1E_2}x$ must be equal to

$$S_{12}=\bigcup_sg_{E_2}(s)\{ g_{E_1}(t)x\,|\, t\in\mathbb R\}=\{ g_{E_2}(s)g_{E_1}(t)x\,|\, t\in\mathbb R, s\in\mathbb R\}, $$
 and fully covers $\Omega$.
 For the case that a one-parameter subgroup $g_{E_1}(t)$ is hyperbolic and another one $g_{E_2}(s)$ is parabolic with {$(F_{\xi_1}^+\cup F_{E_1}^-) \cap F_{E_2} \neq \emptyset$}, we get $S_{12}=H_{E_1E_2}x=\Omega$ since each $g_{E_1}(t)$-orbit meets every $g_{E_2}(s)$-horospheres.

  If the case is not the above two, we can show that
$$S_{123}=\{ g_{E_1}(u)g_{E_2}(s)g_{E_1}(t)x\,|\, t, s, u\in\mathbb R\}$$
is equal to $H_{E_1E_2}x=\Omega$, considering various possibilities. For example, if  $g_{E_1}(t)$ is hyperbolic and $g_{E_2}(s)$ is parabolic with {$(F_{E_1}^+\cup F_{E_1}^-) \cap F_{E_2} = \emptyset$}, then there is a $g_{E_2}(s)$-horosphere, $S_z$,  which meets tangentially an orbit $g_{E_1}(t)x$
and thus $S_{12}$ becomes one of two components of $\Omega-S_z$. This implies that $$\Omega=\bigcup_{u\in\mathbb R}g_{E_1}(u)S_{12}=S_{123}.$$

If the dimension of $\Omega$ is greater than $2$, $H_{E_1E_2}x$ may not be equal to $\Omega$, but $\partial (H_{E_1E_2}x)\subset \partial\Omega$.  In this case, we can find another face $E_3\subset\Lambda_{G^0} $ outside $\overline{H_{E_1E_2}x}$ since $\overline\Omega=\overline{CH({\Lambda_{G^0}})}$ by Lemma \ref{closed}.  Let $H_{E_1E_2E_3}$ be the subgroup of $G^0$ which is generated by $g_{E_1}(t)$,  $g_{E_2}(s)$,  and $g_{E_3}(u)$. Now considering the action of the one parameter subgroup $\{g_{E_3}(u)\}$ we can show that $$\Omega=H_{E_1E_2E_3}x,$$ if the dimension of $\Omega$ is $3$. Actually $\Omega=H_{E_1E_2E_3}x$ is equal to 
$$\bigcup_u g_{E_3}(u)H_{E_1E_2}x \quad \text{or} \quad\bigcup_{h\in H_{E_1E_2}}\bigcup_{u,v} h(v)g_{E_3}(u)H_{E_1E_2}x.$$
For the case that the dimension of $\Omega$ is greater than $3$ and $\Omega\neq H_{E_1E_2E_3}x$,  we can choose another face $E_4\subset\Lambda_{G^0} $ outside $\overline{H_{E_1E_2E_3}x}$. By repeating this process, we can conclude that $G^0x=H_{E_1E_2\dots E_k}x=\Omega$ for some natural number $k\leq n$ and thus  $G^0$ acts on $\Omega$ transitively.
\end{proof}

\begin{Prop}\label{algebraic}
Let $\Omega (\neq \br^n) $ be a convex
quasi-homogeneous affine domain in  $\br^n$ which has an irreducible $G=\Aut(\Omega)$ and a nonempty $\Lambda_{G^0}$.  Then $\Omega$ is a semi-algebraic subset of $\br^n$.
\end{Prop}
\begin{proof}

Note that if $\Omega$ has a complete line then there exist a natural number $k<n$ and a  $n-k$ dimensional properly convex domain $\Omega'$ such that $\Omega=\mathbb R^{k}\times \Omega'$, which comes from the convexity of $\Omega$. We see immediately that
{$G'=\Aut(\Omega')$ is irreducible if $G=\Aut(\Omega)$ is irreducible},
$\Lambda_{G^0}\neq\emptyset$ if and only if $\Lambda_{G'^0}\neq\emptyset$, and
$\Omega$ is a semi-algebraic subset of $\br^n$ if and only if $\Omega'$ is a semi-algebraic subset of $\br^{n-k}$. So we just need to prove our proposition for a properly convex domain $\Omega$.

Goldman and Hirsch proved in Proposition 2.15 of \cite{GH2} that the developing map of a homogeneous affine manifold is a covering onto a semi-algebraic open set. So we can conclude that $\Omega$ is a semi-algebraic subset of $\br^n$, because $\Omega$ is a homogeneous affine domain by Proposition \ref{homo}.
\end{proof}

Now we prove our main theorem.
\begin{proof}[Proof of Theorem \ref{Markus}]
For a compact affine $n$-manifold $M$ with parallel volume, it is well-known that the affine holonomy $\Gamma$ is irreducible and it preserves no (semi)-algebraic subset of $\br^n$. (See \cite{GH1} and \cite{GH2} for references.) So if a compact convex affine manifold $M= \Omega/\Gamma$ has a parallel volume, the affine automorphism group $G=\Aut(\Omega)$ of the developing image $\Omega\subset \br^n$  is irreducible. Hence either $\Omega$ has no boundary, i.e., $\Omega=\br^n$, or $\Lambda_{G^0}$ must be empty by Proposition \ref{homo} and Proposition \ref{algebraic}, which completes the proof for Theorem \ref{Markus}.

\end{proof}


\section{Markus conjecture in dimension $\leq 5$ }\label{dim5}
In the previous section, we proved the Markus conjecture in the affirmative for a convex affine manifold $M$ under the condition of $\Lambda_{\Aut^0(D(\tilde{M}))}\neq \emptyset$.   In this section, we will see that the Markus conjecture is true for dimension $\leq 5$, by proving that the condition of $\Lambda_{\Aut^0(D(\tilde{M}))}\neq \emptyset$ holds if the affine manifold is not complete. To see this we need to prove that 
 every divisible convex affine domain has a nonempty limit set by the identity component of its  affine automorphism group if it is not the whole affine space. 
\begin{theorem}
Let $\Omega$ be a divisible convex affine domain in $\mathbb R^n$ with $n\leq 5$. Then $\Lambda_{\Aut^0(\Omega)}\neq \emptyset$ if $\Omega \neq \mathbb R^n$.
\end{theorem}
\begin{proof}
By Theorem \ref{VeyThm}, every divisible convex affine domain $\Omega$ is a cone if it has no complete line, which implies that $\Lambda_{\Aut^0(\Omega)}\neq \emptyset$ because the cone point must be in $\Lambda_{\Aut^0(\Omega)}$.

If $\Omega$ has a complete line, then $\Omega$ can be decomposed into $\mathbb R^k$ and a properly convex quasi-homogeneous affine domain $\Omega'$ in $\mathbb R^{n-k}$, that is,
$$\Omega=\mathbb R^k \times \Omega'.$$ 
Note that $\Omega'$ may not be divisible, but it is still quasi-homogeneous. Therefore we need to show $\Lambda_{\Aut^0(\Omega')}\neq \emptyset$ to complete the proof. We will see in Propositions \ref{thm-dim2}, \ref{thm-dim3}, \ref{dim4-asy1}, \ref{dim4-asy3}, \ref{dim4-asy2} \ref{dim4-asy4} that  $\Omega'$ is
affinely equivalent  to one of the
following 19 types of domains:
\begin{enumerate}
\item[\rm(i)] $\{x \in \mathbb R\,|\, x>0\}$
\item[\rm(ii)] $\{(x,y) \in \mathbb R^2\,|\, x>0,\ y>0\}$
\item[\rm(iii)] $\{(x,y) \in \mathbb R^2\,|\, y>x^2 \}$
\item[\rm(iv)] $\{(x,y,z) \in \mathbb R^3\,|\, z>x^{2} +y^{2}
\}$
\item[\rm(v)] $\{(x,y,z) \in \mathbb R^3\,|\, y>x^2,\ z>0\}$
\item[\rm(vi)] $\{(x,y,z) \in \mathbb R^3\,|\, x>0,\ y>0,\ z>0
\}$
\item[\rm(vii)]a 3-dimensional elliptic cone,
\item[\rm(viii)]a 3-dimensional  non-elliptic strictly convex cone,
\item[\rm(ix)] $\{(x_1,x_2,x_3,x_4) \in \mathbb R^4\,|\, x_4>x_1^{2}
+x_2^{2}+x_3^{2} \}$
\item[\rm(x)] $\{(x_1,x_2,x_3,x_4) \in \mathbb R^4\,|\, x_2>x_1^{2},\
 x_3>0,\ x_4>0 \}$
\item[\rm(xi)]
 $\{(x_1,x_2,x_3,x_4) \in \mathbb R^4\,|\, (x_2-x_1^{2})x_3>x_4^{2} \}$
\item[\rm(xii)] 
$\{(x_1,x_2,x_3,x_4) \in \mathbb R^4\,|\, x_3>x_1^{2}
       +x_2^{2},\ x_4>0\}$
\item[\rm(xiii)] $\{(x_1,x_2,x_3,x_4) \in \mathbb R^4\,|\, x_2>x_1^{2}
,\ x_4>x_3^{2} \}$
\end{enumerate}
and 
\begin{enumerate}
\item[\rm(xiv)]a 4-dimensional elliptic cone,
\item[\rm(xv)]a 4-dimensional  non-elliptic strictly convex cone,
\item[\rm(xvi)] a double cone over a triangle, i.e.,$$\{(x_1,x_2,x_3,x_4) \in \mathbb R^4\,|\, x_i>0 \text{ for }
i=1,2,3,4 \},$$
\item[\rm(xvii)]a double cone over an ellipse,
\item[\rm(xviii)]a double cone over a non-elliptic strictly convex
domain,
\item[\rm(xix)]a cone over a 3-dimensional non-strictly convex
indecomposable projective domain.
\end{enumerate}
(i) is for $n-k=1$, (ii) and (iii) are for $n-k=2$, (iv) and (viii) are for $n-k=3$, (ix)-(xiii) and (xiv)-(xix) are for $n-k=4$.
Each of the above 19 types of domains  is either homogeneous or a cone. If $\Omega$ is homogeneous then 
$\Lambda_{\Aut^0(\Omega')}= \partial\Omega'$, and if $\Omega$ is a cone then its cone point is in the limit set $\Lambda_{\Aut^0(\Omega')}$.
\end{proof}

\begin{Prop}[\cite{Ben}, \cite{Jo}]\label{thm-dim2}Let $\Omega$ be a properly convex
quasi-homogeneous affine domain in $\mathbb R^2$. Then $\Omega$ is
affinely equivalent to either a quadrant or a parabola.
\end{Prop}

\begin{Prop}[Jo, \cite{Jo}]\label{thm-dim3}Let $\Omega$ be a properly convex quasi-homogeneous
 affine domain in $\mathbb R^3$. Suppose $\Omega $ is
not a cone. Then $\Omega $ is affinely equivalent to one of the
following:
\begin{enumerate}
\item[\rm(i)] A 3-dimensional paraboloid, i.e., $$\Omega=\{(x,y,z) \in \mathbb R^3\,|\, z>x^{2} +y^{2}
\}.$$
\item[\rm(ii)] A  parabola $\times \mathbb R^{+}$,
 i.e.,
 \begin{equation*}
\begin{split}
  \Omega=&\{(x,y) \in \mathbb R^2\,|\,y>x^{2}\}\times \{z \in \mathbb R\,|\,
  z>0\}\\
  =& \{(x,y,z) \in \mathbb R^3\,|\, y>x^2,\ z>0\}.
\end{split}
\end{equation*}
\end{enumerate}
\end{Prop}

\begin{Prop}[Jo, \cite{Jo}]\label{dim4-asy1}Let $\Omega$ be a properly convex quasi-homogeneous
 affine domain in $\mathbb R^4$ with 1-dimensional asymptotic cone. Then $\Omega $
 is a 4-dimensional paraboloid, i.e.,
 $\Omega$ is affinely equivalent to $\{(x_1,x_2,x_3,x_4) \in \mathbb R^4\,|\, x_4>x_1^{2}
+x_2^{2}+x_3^{2} \}$.
\end{Prop}

\begin{proof}
This is an immediate consequence of Theorem 5.9 in \cite{Jo}.

\end{proof}

\begin{Prop}\label{dim4-asy3}Let $\Omega$ be a properly convex quasi-homogeneous
 affine domain in $\mathbb R^4$ with 3-dimensional asymptotic cone. Then $\Omega $ is affinely equivalent to one of the
following:
\begin{enumerate}
\item[\rm(i)]
$\{(x_1,x_2,x_3,x_4) \in \mathbb R^4\,|\, x_2>x_1^{2},\
 x_3>0,\ x_4>0 \}$
\item[\rm(ii)]
 $\{(x_1,x_2,x_3,x_4) \in \mathbb R^4\,|\, (x_2-x_1^{2})x_3>x_4^{2} \}$
\end{enumerate}
\end{Prop}
\begin{proof}
The set of asymptotic cone points $S$ is a curve in $\mathbb R^4$
by Theorem \ref{thm-folliation} because AC$(\Omega)$ is
$3$-dimensional. Let $S_{\infty}$ be the set of limit points of
$S$ in $\mathbb {RP}^4$. Then $S_{\infty}$ consists of at most
$2$-points. Since $\Aut(\Omega)$ acts on $S$ transitively, $S_{\infty}$ must be contained in the infinite boundary.

Consider the convex hull CH$(S)$ of $S$. Since all points in $S$
are extreme points, the dimension of the minimal projective
subspace $V$ containing $S$ is greater than $1$. Let $F$ be the
relative interior of the closure of CH$(S)$. Then $F$ is an
invariant face of $\Omega$ with dimension greater than $1$.

By \cite{Jo1}, AC$(\Omega)$ is quasi-homogeneous. So there are two
cases: AC$(\Omega)$ is either a simplex cone or a strictly convex
cone.

\begin{enumerate}
\item[\rm(i)]  Assume that $$\AC(\Omega)=\{(x_1,x_2,x_3) \in \mathbb R^3 \,|\,x_1>0,\
x_2>0,\ x_3>0\}.$$

Let $\{a, b, c\}$ be the set of infinite extreme points of
$\Omega$. Then AC$(\Omega)$ is a tetrahedron with vertices $a, b,
c$ and the origin in $\mathbb R^4$, when we consider it as a
projective domain in $\rp{4}$. We denote the infinite face of
$\Omega$ by $\triangle_{abc}$. If we let $G$ be the set of all the element of  $\Aut(\Omega)$ fixing each of the infinite extreme points $a,b,c$, then $G$ also acts on $\Omega$ syndetically. 

 First, we show that $S_{\infty}$ cannot have an interior
point of $\triangle_{abc}$. This follows from the fact that the
projectivization of the linear parts of $G<\Aut(\Omega)$ acts on the infinite boundary of $\Omega$ syndetically and the fact that if $G$ acts on a domain $D$ syndetically then $G$-invariant set cannot intersect $D$.

Now we know that $S_{\infty} \subset \partial\triangle_{abc}$. If
$s'$ is a point in $S_{\infty}$ which is not in $\{a, b, c\}$, say
$s'$ is in the open line segment $(a,b)$, then another point in
$S_{\infty}$ should be also in the closed line segment $[a,b]$ if
exists. Suppose $s''\in S_{\infty}$ is not in $[a,b]$. Then the
line $(s',s'')$ is preserved by the action of the projectivization
of the linear parts of $\Aut(\Omega)$, which is again a
contradiction. Hence we may assume that $S_{\infty}\subset [a,b]$. 

Since $[a,b]$ is
preserved under the action of $G$, there is
a sequence of projective transformation $g_i\in G<
\Aut_\text{proj}(\Omega)$ whose limit singular projective
transformation is $g$ and the range of $g$ is one point $c$. Then the
dimension of K$(g)=3$, and K$(g)$ contains $[a, b]$ and $S$. If we
let the relative interior of the convex hull of $S\cup[a,b]$ by
$E$, then $E$ should be contained in K$(g)\cap
\partial\Omega$, so $E$ is a face of $\Omega$.
Since $c$ is the only extreme point of $\Omega$ which is not
contained in the closure of $E$, $\overline\Omega$ is the convex
hull of $\overline E\cup \{c\}$. This implies that
$$\Omega= E\dot{+}\{c\}.$$
Now we will show that $E$ is affinely equivalent to  a parabola
$\times \mathbb R^{+}$. In fact, since $E$ is a $3$-dimensional conic face with $2$-dimensional asympototic cone and any conic
face is quasi-homogeneous,  $E$ is affinely equivalent to  a parabola
$\times \mathbb R^{+}$ by Proposition
\ref{thm-dim3}. So $\Omega$ is affinely equivalent to
 \begin{equation*}
\begin{split}
  E\dot{+}\{c\}=&\{(x_1,x_2,x_3) \in \mathbb R^3\,|\, x_2>x_1^{2},\
 x_3>0\}\dot{+}\{c\}\\
 =&\{(x_1,x_2,x_3,x_4) \in \mathbb R^4\,|\, x_2>x_1^{2},\
 x_3>0,\ x_4>0 \}.
\end{split}
\end{equation*}
\item[\rm(ii)] Assume that AC$(\Omega)$ is a strictly convex cone.
In this case $S_{\infty}$ cannot have two points: If so, a
quasi-homogeneous cone AC$(\Omega)$ has an invariant proper
subset, which cannot happen by Proposition \ref{Vey}. So we may assume that
there is a point $\xi$ in the infinite boundary of $\Omega$ which
is the unique limit point of the set of all extreme points in
$\mathbb R^4$. Since the infinite boundary of $\Omega$ is a
strictly convex quasi-homogeneous projective domain with unique
fixed point $\xi$ (under the projective action of
$\Aut(\Omega)$ on the infinite boundary), it is an
ellipse by Proposition \ref{strictlyCV}.

Choose a boundary point $a$ of the infinite boundary of $\Omega$
which is different from $\xi$ and consider a sequence of
projective transformation $g_i\in \Aut(\Omega)<\Aut_\text{proj}(\Omega)$ whose
limit singular projective transformation is $g$ and the range of $g$
is $\{a\}$. Then $S$ and $\xi$ is in $K(g)$, which shows $\overline{{\text {CH}}(S)}\subset K(g)$. So the dimension of $F$ cannot be $4$, i.e.,
$\overline{\text{CH}(S)}\neq\overline\Omega$, since $K(g)\cap \Omega=\emptyset$.

If the dimension of $F$ is $3$, then $F$ is a conic face and thus
there is an (infinite) boundary point $z$ such that
$$\Omega=F\dot{+}\{z\},$$ which contradicts that $\mbox{AC}(\Omega)$ is an
elliptic cone.

Since the dimension of $F$ is greater than $1$, it is $2$, and so
$F$ is a strictly convex $2$-dimensional invariant face of
$\Omega$. Note that $\partial F=S$ is twice-differentiable, since the smooth Lie group $\Aut(\Omega)$  acts on $S$  transitively. By Lemma \ref{Vey}, there is a sequence of affine transformation $\{g_i\}\subset  \Aut(\Omega)$ converging to a singular projective transformation $g$ whose range is an extreme point $p$. Since the kernel of $g$ does not intersect $\mathbb R^4$ by Lemma \ref{kernel}, $g_i(x)$ converges to $p$ for any $x\in \mathbb R^4$. Since $F$ is an invariant face of $\Omega$, the restriction of $g_i$ on the support of $F$  is an element of $\Aut(F)$ and $g_i(x) \in F$ converges to $p \in \partial F$ for $x\in F$.
This implies that $F$ is a parabola because a strictly convex domain with a twice differentiable boundary is  a parabola if an orbit accumulates at a boundary point. See \cite{Jo4}. 

We have shown up to now that $F$ is a parabola and its infinite boundary is  a fixed point $\xi \in \partial_{\infty} \Omega$. 
So we can choose a basis $\{\bf{x_1},\bf{x_2},\bf{x_3},\bf{x_4}\}$ of $\mathbb R^4$ 
such that $\bf{x_2}$ is the invariant direction going to $\xi$, that is, $\lim_{r\rightarrow \infty}r\bf{x_2}=\xi$,
$$F=\{(x_1,x_2,0,0) \in \mathbb R^4\,|\, x_2>x_1^{2}\}.$$  and
$$ \text{AC}(\Omega)=\{(0,x_2,x_3,x_4) \in \mathbb R^4\,|\, x_2x_3\geq x_4^{2} \},$$
Since $\bar \Omega$ is the union of all the asymptotic cones whose cone points are  extreme points of $\Omega$, that is,
 $$\Omega=\bigcup_{s\in S} s+\text{AC}^\circ(\Omega ).$$
and each ray $s+r{\bf x_2} \subset \bar F$ for $s\in S$, we can conclue that 
$$\Omega=\{(x_1,x_2,x_3,x_4) \in \mathbb R^4\,|\, (x_2-x_1^{2})x_3>x_4^{2} \}.$$
\end{enumerate}
\end{proof}
\begin{Prop}\label{dim4-asy2}Let $\Omega$ be a properly convex quasi-homogeneous
 affine domain in $\mathbb R^4$ with 2-dimensional asymptotic cone. Then $\Omega $ is affinely equivalent to one of the
following:
\begin{enumerate}
\item [\rm (i)] A paraboloid $\times \mathbb
R^{+}$,
 i.e.,
 \begin{equation*}
\begin{split}
  &\{(x_1,x_2,x_3) \in \mathbb R^3\,|\,x_3>x_1^{2}
+x_2^{2}\}\times \{x_4 \in \mathbb R\,|\, x_4>0\}\\
&=\{(x_1,x_2,x_3,x_4) \in \mathbb R^4\,|\, x_3>x_1^{2}
       +x_2^{2},\ x_4>0\}.
\end{split}
\end{equation*}
\item [\rm (ii)]A parabola $\times $ a parabola :$$\{(x_1,x_2,x_3,x_4) \in \mathbb R^4\,|\, x_2>x_1^{2}
,\ x_4>x_3^{2} \}.$$
\end{enumerate}
\end{Prop}
\begin{proof}
 In this case $S$ is a 2-dimensional hypersurface in $\mathbb R^4$ and $\Omega$ has two extreme points in the infinite
boundary. Let $z$ and $w$ be the infinite extreme points of
$\Omega$. Then $AC(\Omega)$ is a triangle with three vertices $z,
w$ and the origin. The dimension of $F=CH(S)^{\circ}$ is either 3 or 4. Let $G$ be the set of all the element of  $\Aut(\Omega)$ fixing both $z$ and $w$.  Then $G$ also acts on $\Omega$ syndetically and thus acts on $S$ transitively by Theorem \ref{transitive}.

(i) If dim$F=3$, then $F$ is a conic face of $\Omega$ and thus there is a face $B$ of $\Omega$ such that 
$\Omega=F\dot{+}B$ by Theorem \ref{section-benz}, (ii). So one of $z$ or $w$ is in $\overline F$ and the other is in $\overline B$ by Corollary \ref{AC-Face}, (ii). Suppose  that $w$
 is in $\overline F$. Then we see that  $F$
 is a 3-dimensional quasi-homogeneous affine domain with
 a 1-dimensional asymptotic cone by Theorem \ref{prop-benz}, (iii).
 By Proposition \ref{thm-dim3}, $\partial F$ is a paraboloid and as a projective domain in $\mathbb {RP}^4$
 $$\Omega=F\dot{+}B=F\dot{+}\{z\}.$$ Therefore $\Omega$ is affinely equivalent
 to $\{(x_1,x_2,x_3) \in \mathbb R^3\,|\,x_3>x_1^{2}
+x_2^{2}\}\times \{x_4 \in \mathbb R\,|\, x_4>0\}.$

(ii) If dim$F=4$, then $\Omega=F$ and thus $\overline\Omega=\overline{CH(S)}$.
Note that $\Omega$ cannot have any conic face in this case by (iii) of Corollary \ref{AC-Face}.
Let $\bf z$ and $\bf w$ be the unit vectors of $\mathbb R^4$
in the direction of $z$ and $w$ respectively.
 Suppose that all the rays $\{e+ r{\bf
w}\,|\, r>0, e\in S \}$ are one dimensional faces of $\Omega$. Choose a $3$ dimensional affine space $V$ in $\mathbb R^4$ which is transversal to $\bf w$, and project $\Omega$ into $V$ along $\bf w$. If we denote the projection by $p:\Omega\rightarrow V$, then $p(\Omega)\subset V$ is a $3$-dimensional  convex quasi-homogeneous affine domain.
Since every extreme point $e$  of $\Omega$  is projected to an extreme point  $p(e)$ of $p(\Omega)$ if $\{e+ r{\bf
w}\,|\, r>0\}$ is a one dimensional face,  $p(\Omega)$ cannot contain any complete line. (For,  if there is a complete line $l$ in  $\overline{p(\Omega)}$ then $p(e)+l$ should be also contained in  $\overline{p(\Omega)}$ by convexity.)
So there are linearly independent  {supporting} hyperplanes {of $p(\Omega)$,} $H'_1,H'_2,H'_3$, in $V$.
If we denote the hyperplanes of $\mathbb R^4$ which contain $p^{-1}(H'_1),p^{-1}(H'_2),p^{-1}(H'_3)$ by $H_1,H_2,H_3$ respectively in $\mathbb R^4$, then 
$H_1,H_2,H_3$ are linearly independent supporting hyperplanes of  $\Omega$. Now if we see $\Omega$ in $\mathbb {RP}^4$, then $H_1,H_2,H_3$ and $\partial_\infty \mathbb R^4$ are linearly independent supporting hyperplanes of the extreme point $w$. This  implies $w$ is a conic face of $\Omega$, which is a contradiction. 
In a similar manner, we can show the same thing for $z$.

So we can choose
 faces $A$ and $B$ whose closures contain properly a ray $\{\xi+
r{\bf w}\,|\, r>0\}$ and a ray $\{\xi'+ r{\bf z}\,|\, r>0\}$
respectively for some $\xi, \xi'\in S$.  {If either $A$ or $B$ is 3-dimensional, then it is a conic face, which is not allowed. Hence both $A$ and $B$ are
properly convex $2$-dimensional affine domains with a 1-dimensional
asymptotic cone.} This implies that every asymptotic cone point $\xi$ is contained in the boundary of two 2-dimensional faces $A_{\xi}$ and $B_{\xi}$ whose asymptotic directions are $w$ and $z$  respectively, since
the subgroup $G$ of $\Aut(\Omega)$ acts on $S$ transitively.

For each extreme point  $\xi$, we see that there does not exist any other proper face except $A_{\xi}$ and $B_{\xi}$ whose cloure contains $\{\xi\}$ properly, because any line segment in $\partial\Omega$ must be contained in the closure of a face having a ray with direction $\bf z$ or $\bf w$.  So  $A_{\xi}$ and $B_{\xi}$ are maximal and the only  faces of $\Omega$  whose dimension is neither $0$ nor $4$. By Proposition \ref{bdd-face} and Corollary \ref{AC-Face}, $\partial A_{\xi} \subset S$ and thus  $A_{\xi}$ is a strictly convex $2$-dimensional face of $\Omega$ whose boundary is a part of $S$.

Since there are no other nonzero-dimensional boundary faces of $\Omega$  except $A_{\xi}$'s and $B_{\xi}$'s,  either $A_{\xi}\cap A_{\xi'}=\emptyset$ or $A_{\xi}= A_{\xi'}$ holds and the same is true for $B_{\xi}$'s. Note that $A_{\xi}= A_{\xi'}$ if and only if $\xi' \in \partial A_{\xi}$. Let  $S_{\xi}$ be a subset of $S$ defined as follows: $$S_{\xi}=\bigcup_{\zeta\in\partial A_{\xi}} \partial B_{\zeta}$$
Then $S_{\xi}\cap S_{\xi'}=\emptyset$ for $\xi' \notin S_{\xi}$. Therefore $S = S_{\xi}$ for any extreme point  $\xi$ and thus   $\partial B_{\xi_1}\cap \partial A_{\xi_2} \neq \emptyset$ for any pair of extreme points  $\xi_1$  and $\xi_2$, since $S$ is connected.
Similary one can prove $$S=\bigcup_{\eta\in\partial B_{\xi}} \partial A_{\eta}.$$

If we consider  the subgroups $G_{A_{\xi}}$ ($G_{B_{\xi}}$,  respectively) of $G$ preserving $A_{\xi}$ ($B_{\xi}$,  respectively), then $G_{A_{\xi}}$($G_{B_{\xi}}$,  respectively) acts on $\partial A_{\xi}$($\partial B_{\xi}$,  respectively) transitively and thus the boundary of $A_{\xi}$ and  $B_{\xi}$ are both twice-differentiable, which means each of $A_{\xi}$ and  $B_{\xi}$ has an osculating ellipsoid at every boundary point by Remark \ref{Ben-oscullating}. There are two continuous maps $\rho_A$ and $\rho_B$ from $\mathbb R$ to $G_{A_{\xi}}$ and $G_{B_{\xi}}$, respectively, such that  $\{\rho_A(t)(\xi)\,|\, t\in \mathbb R\}=\partial A_{\xi}$ and $\{\rho_B(t)(\xi)\,|\, t\in \mathbb R\}=\partial B_{\xi}$.
 Then each of  $\rho_A(t)$-orbits and $\rho_B(t)$-orbits gives a foliation on $\overline\Omega \cap \mathbb R^4$ whose leaves are homeomorphic to $\partial A_{\xi}$ and $\partial B_{\xi}$, respectively.

Now we claim here that $A_{\xi}$ is a parabola. Since $\Omega$ is quasi-homogeneous and $\xi$ is an extreme point, there is a sequence of affine transformation $\{g_i\}\subset G$ such that the range of its limit singular projective transformation $g$ is $\{\xi\}$ by Lemma \ref{lem-saillant}. Since the kernel of $g$ does not intersect $\mathbb R^n$ by Lemma \ref{kernel},
$g_i(p)$ converges to $\xi$ for any point $p$ in $A_{\xi}$. Let $\xi'$ be the boundary point of $B_{\xi}$ such that $g_i(p)\in A_{\xi'}=A_{g_i(\xi)}$ and $g_i(\xi)\in \partial A_{\xi'}$ (the existence of $\xi'$ is guaranteed by the property $S=\bigcup_{\eta\in\partial B_{\xi}} \partial A_{\eta}$). 
Then for each $g_i$ 
there is $t_i\in \mathbb R$ such that $\rho_B(t_i)(\xi)=\xi'$. Then $\rho_B(t_i)^{-1}g_i(p)$ is in $A_{\xi}$  for all $i$.
 Since the sequence of  leaves $\{\rho_B(t)(\rho_B(t_i)^{-1}g_ip)\,|\, t\in \mathbb R\}=\{\rho_B(t)(g_ip)\,|\, t\in \mathbb R\}$ converges to $\{\rho_B(t)(\xi)\,|\, t\in \mathbb R\}=\partial B_{\xi}$, the leaf containing $\xi$, $\rho_B(t_i)^{-1}g_i(p)$ must converge to $\xi$. So we found a sequence of projective transformations $f_i=\rho_B(t_i)^{-1}g_i$ preserving $A_{\xi}$ such that an orbit $f_i(p)$ for $p\in A_{\xi}$  converge to $\xi \in \partial A_{\xi}$.
This implies that $A_{\xi}$  is a parabola due to Theorem \ref{hess-noncompt} stating that a strictly convex domain with a twice differentiable boundary is a parabola if there is an orbit accumulating at a boundary point.

Similarly, we can show that $B_{\xi}$ is also a parabola, by taking a sequence of projective transformations $\{\rho_A(t_i)\}\subset G_{A_{\xi}}$ such that $\rho_A(t_i)^{-1}g_i(B_{\xi})=B_{\xi}$ and  $\rho_A(t_i)^{-1}g_i(p')$ converges to $\xi$ for $p'\in B_{\xi}$.

Now we will show that $\Omega$ is affinely equivalent to $$\{(x_1,x_2,x_3,x_4) \in \mathbb R^4\,|\, x_2>x_1^{2}
,\ x_3>x_4^{2} \},$$ by proving that  any other $A_{\xi'}$ is obtained from $A_{\xi}$ by just an affine translation.

We first show that the isotropy subgroup $G_{\xi}<G$ fixing $\xi$ is not trivial. If we suppose that $G_{\xi}$ is trivial, then $G$ acts on $S$ simply transitively and thus $G$ is connected, which implies that $\Omega$ is homogeneous by Lemma 2.5 of \cite{Bt}. But the isotropy subgroup of a homogeneous properly convex affine domain cannot be trivial because any  element of $G$ sending a point of $\xi+AC(\Omega)$ to another point of the cone fixes $\xi$, which contradicts our assumption.   

Choose two vectors $\bf a$ and $\bf b$ such that
$\xi+{\bf a}$ is in the support of $A_{\xi}$ and $ \xi+{\bf b}$ is in the support of $ B_{\xi}$,
and $\{\bf a,\bf w,\bf z,\bf b\}$ is a basis for $\mathbb R^4$.
Then  there is a non-trivial element $g \in G_{\xi}$ such that 
$$g=\begin{pmatrix}
\delta & 0 & 0 & 0 \\
0 &\delta^2 & 0 & 0 \\
0 & 0 & \theta^2 & 0  \\
0 & 0 & 0 & \theta
\end{pmatrix}$$
with $\delta >0$ and $\theta >0$, if we normalize $\xi=(0,0,0,0)$ and the parabolas $\partial A_{\xi}$ and $\partial B_{\xi}$ by $x_2=x_1^2$ and $x_3=x_4^2$ respectively, since $g$ must preserve both $A_{\xi}$ and $B_{\xi}$. 

If $\delta=1$, then
$g\in G_{\zeta}$ for all $\zeta \in \partial A_{\xi}$, that is, 
\begin{equation}\label{delta}
g(\zeta)=\zeta, \quad g(B_{\zeta})=B_{\zeta}  \quad \text{for all}\quad \zeta \in \partial A_{\xi}.
\end{equation}  
So we can prove that $B_{\zeta}$ must be just a translation of  $B_{\xi}$ for any $\zeta \in \partial A_{\xi}$. Suppose not. Then there is $\zeta \in \partial A_{\xi}$ such that for some $r> 0$  and $c, d_1, d_2 \in \mathbb R$
$$\partial B_{\zeta}=\{(x_1,x_2,x_3,x_4) \in \mathbb R^4\,|\, x_3=r(x_4-c)^{2}-rc^2,\ x_4=d_1x_1+d_2x_2 \}+\zeta,$$ 
which is not equal to 
$${\partial B_{\xi}+\zeta}=\{(x_1,x_2,x_3,x_4) \in \mathbb R^4\,|\, {x_1=x_2=0, \,\,}x_3=x_4^{2}\}{+\zeta}.$$ 
But this implies that 
\begin{equation*}
	\begin{split}
		g(\partial B_{\zeta})=\{(x_1,x_2,x_3,x_4) \in \mathbb R^4\,|\, x_3=r(x_4-\theta c)^{2}-r\theta^2c^2,\ x_4=\theta(d_1x_1+d_2x_2) \}+\zeta,	
	\end{split}
\end{equation*} which contradicts (\ref{delta}).
Similarly, we also get that $A_{\zeta}$ must be just a translation of  $A_{\xi}$ for any $\zeta \in \partial B_{\xi}$ if $\theta=1$. Hence we may assume that  $\delta \neq 1$ and $0<\theta < 1$ by taking the inverse $g^{-1}$ if necessary.

Since $G_{B_{\xi}}$ is a closed Lie subgroup of $G$ and $G_{B_{\xi}}(\xi)=\partial B_{\xi}$, we can take an element $\eta $ in the Lie-subalgebra  $\mathfrak{g}_{B_{\xi}}$ of $\mathfrak{g}^0$  such that
the one parameter group $f_{t}=e^{t\eta}$ is non-trivial.	

\noindent {\bf Case 1. $f_{t}$ is hyperbolic:}
	 If $f_{t}$ is hyperbolic, then  we may assume that   $f_{t}(\xi)=\xi$ and thus we get  a one-parameter subgroup $g(t)$ of  $G_{\xi}$, 
$$f_{t}=g(t)=\begin{pmatrix}
e^{t\mu} & 0 & 0 & 0 \\
0 & e^{2t\mu} & 0 & 0 \\
0 & 0 & e^{2t\nu} & 0  \\
0 & 0 & 0 & e^{t\nu}
\end{pmatrix},$$
and another element $g'$ of $G_{B_{\xi}}$ such that $g'(\xi)=(0,0,1,1)$ by the transitivity of the action of $G$ on $S$. 
We can choose an element $f$ of $G_{B_{\xi}}$ among  $g(t)g'$ acting on the face $B_{\xi}$ as a parabolic transformation as in the following lemma. 

\begin{lemma}The linear part and translation part
of $f$ can be represented by 
$$L_f=\begin{pmatrix}
\alpha_1 & 0 & 0 & 0 \\
\beta_1 &\alpha_2^2 & 0 & 0 \\
\beta_2 & 0 & 1 & 2d  \\
\beta_3 & 0 & 0 & 1
\end{pmatrix},\quad
t_{f}=\begin{pmatrix}
0 \\
0\\
d^2 \\
d
\end{pmatrix}$$
for some positive real number $d$.
\end{lemma}
\begin{proof}Since $g'$ preserves $B_\xi$ and $\bf w, \bf z$ directions,
its linear part is of the form 
$$\begin{pmatrix}
* & 0 & 0 & 0 \\
* & * & 0 & 0 \\
* & 0 &  *  & * \\
* & 0 & 0 & * \end{pmatrix}$$
Then by choosing $t$ properly, the $(3,3)$ component of linear part $L_f$ of the product $f=g(t)g'$ is 1.
Since $f$ also preserves  $B_\xi$ and $\bf w, \bf z$ directions
$$L_f=\begin{pmatrix}
* & 0 & 0 & 0 \\
* & * & 0 & 0 \\
* & 0 &  1  & a \\
* & 0 & 0 & b \end{pmatrix}$$
Since $g'$ sends $\xi$ to $(0,0,1,1)$ the translation part $t_f$ of $f$ is 
$(0,0, e^{2t\nu}, e^{t\nu})=(0,0,d^2,d)$ for some $d$.
Now using $f(\partial B_\xi)=\partial B_\xi=\{(0,0,x^2,x)\}$ we get
$$ X_3= x^2+ ax +d^2, X_4=bx + d, X_4^2=X_3.$$
Hence $b=1, a=2d$ as desired.
\end{proof}

If $\mu, \nu>0$ or $\mu, \nu <0$ then $\xi\in \Lambda_{G^0}$ and thus $\Omega$ is homogeneous  by Proposition \ref{homo}, {since $\Aut(\Omega)$ is irreducible by (iii) of Corollary \ref{AC-Face}}, which implies that $\Omega$ must be  affinely equivalent to $$\{(x_1,x_2,x_3,x_4) \in \mathbb R^4\,|\, x_2>x_1^{2}
	,\ x_3>x_4^{2} \}$$
by the well-known classification of homogeneous projective convex domains in $\mathbb R^4$. (See p.282 of \cite{Rogers} for a reference.)
	 So we may assume that $\mu>0$ and $\nu <0$. Then  by considering 
\begin{equation}\label{fn}
	 g(t_n)f(\xi)=f^n(\xi)=(0,0,n^2d^2, nd)
	 \end{equation}
	 and
	 \begin{equation}\label{f-n}
	  g(t_n)f^{-1}(\xi)=f^{-n}(\xi)=(0,0,n^2d^2, -nd)
	 \end{equation}
	  for the sequence $t_n=\frac{\ln n}{\nu}, n>0$,
	 we  can show that 
	 $$\alpha_1^2=\alpha_2^2, \quad \beta_1=\beta_2=\beta_3=0$$
	 as follows.  
Since the linear part of $f^n$ is
$$L_{f^n}=\begin{pmatrix}
\alpha_1^n & 0 & 0 & 0 \\
\beta_1^* &\alpha_2^{2n} & 0 & 0 \\
\beta_2^* & 0 & 1 & 2nd  \\
\beta_3^* & 0 & 0 & 1
\end{pmatrix}$$  
where
\begin{equation*} 
\begin{split}
\beta_1^*&=\beta_1(\alpha_1^{n-1}+\alpha_1^{n-2}\alpha_2^2+\cdots+\alpha_1\alpha_2^{2n-4}+\alpha_2^{2n-2})\\
\beta_2^*&=\beta_2(1+\alpha_1+\alpha_1^2+\cdots+\alpha_1^{n-1})+2d\beta_3(n-1+(n-2)\alpha_1+\cdots+2\alpha_1^{n-3}+\alpha_1^{n-2})\\
\beta_3^*&=\beta_3(1+\alpha_1+\alpha_1^2+\cdots+\alpha_1^{n-1}),
\end{split}
\end{equation*}	
we get 
\begin{equation*} 
f^n(x,x^2,0,0)=(\alpha_1^nx,\beta_1^*x+\alpha_2^{2n}x^2,\beta_2^*x,\beta_3^*x)+(0,0,n^2d^2,nd) \\
\end{equation*}	
and thus  $f^n(\partial A_{\xi})$ are the set of all the points $(X_1,X_2,X_3,X_4)+(0,0,n^2d^2,nd) $ such that
\begin{equation*}
\begin{split}
X_2&=\frac{\beta_1}{\alpha_1}(1+\frac{\alpha_2^2}{\alpha_1}+\frac{\alpha_2^4}{\alpha_1^2}+\cdots+\frac{\alpha_2^{2n-2}}{\alpha_1^{n-1}})X_1+\frac{\alpha_2^{2n}}{\alpha_1^{2n}}X_1^2\\
X_3&=\frac{\beta_2}{\alpha_1}(1+\frac{1}{\alpha_1}+\frac{1}{\alpha_1^2}+\cdots+\frac{1}{\alpha_1^{n-1}})X_1+\frac{2d\beta_3}{\alpha_1}(\frac{1}{\alpha_1}+\frac{2}{\alpha_1^2}+\cdots+\frac{n-1}{\alpha_1^{n-1}})X_1\\
X_4&=\frac{\beta_3}{\alpha_1}(1+\frac{1}{\alpha_1}+\frac{1}{\alpha_1^2}+\cdots+\frac{1}{\alpha_1^{n-1}})X_1.
\end{split}
\end{equation*}	
On the other hand, the elements of $g(t_n)f(\partial A_{\xi})$ are
\begin{equation*}
\begin{split} 
g(t_n)f(x,x^2,0,0)&=(Y_1,Y_2,Y_3,Y_4)+(0,0,n^2d^2,nd) \\
Y_2&=\frac{\beta_1}{\alpha_1}e^{t_n\mu} Y_1+\frac{\alpha_2^2}{\alpha_1^2}Y_1^2\\
Y_3&=\frac{\beta_2}{\alpha_1}e^{t_n(2\nu-\mu)} Y_1\\
Y_4&=\frac{\beta_3}{\alpha_1}e^{t_n(\nu-\mu)} Y_1.
\end{split}
\end{equation*}	
Since $g(t_n)f(A_{\xi})$ equals $f^n(A_{\xi})$ by (\ref{fn}), 
we see
$\alpha_1^2=\alpha_2^2 $, 
and if $\beta_1\neq 0$
$$e^{t_n\mu}=1+\alpha_1+\alpha_1^2+\cdots+\alpha_1^{n-1}.$$ But $\lim_{n \to \infty}e^{t_n\mu}=0$ and $1+\alpha_1+\alpha_1^2+\cdots+\alpha_1^{n-1}$ cannot converges to $0$, which implies $\beta_1=0$. 

{Now we get $$L_f=\begin{pmatrix}
\alpha_1 & 0 & 0 & 0 \\
0 &\alpha_1^2 & 0 & 0 \\
\beta_2 & 0 & 1 & 2d  \\
\beta_3 & 0 & 0 & 1
\end{pmatrix},\quad
t_{f}=\begin{pmatrix}
0 \\
0\\
d^2 \\
d
\end{pmatrix}$$
and
$$L_{f^{-1}}=\begin{pmatrix}
1/\alpha_1 & 0 & 0 & 0 \\
0 &1/\alpha_1^2 & 0 & 0 \\
(2d\beta_3-\beta_2)/\alpha_1 & 0 & 1 & -2d  \\
-\beta_3 /\alpha_1& 0 & 0 & 1
\end{pmatrix},\quad
t_{f}=\begin{pmatrix}
0 \\
0\\
d^2 \\
-d
\end{pmatrix}.$$
Thus
$$L_{f^{-n}}=\begin{pmatrix}
1/\alpha_1^n & 0 & 0 & 0 \\
0 &1/\alpha_1^{2n} & 0 & 0 \\
\tilde{\beta}_2 & 0 & 1 & -2nd  \\
\tilde{\beta}_3 & 0 & 0 & 1
\end{pmatrix}, \ n>0$$  
where
\begin{equation*} 
\begin{split}
\tilde{\beta}_2&=-\frac{\beta_2}{\alpha_1}(1+\frac{1}{\alpha_1}+\frac{1}{\alpha_1^2}+\cdots+\frac{1}{\alpha_1^{n-1}})+\frac{2d\beta_3}{\alpha_1}(n+\frac{n-1}{\alpha_1}+\frac{2}{\alpha_1^2}+\cdots+\frac{1}{\alpha_1^{n-1}})\\
\tilde{\beta}_3 &=-\frac{\beta_3}{\alpha_1}(1+\frac{1}{\alpha_1}+\frac{1}{\alpha_1^2}+\cdots+\frac{1}{\alpha_1^{n-1}}).
\end{split}
\end{equation*}
So we have
\begin{equation*} 
f^{-n}(x,x^2,0,0)=(\frac{1}{\alpha_1^n}x,\frac{1}{\alpha_1^{2n}}x^2,\tilde{\beta}_2x,\tilde{\beta}_3 x)+(0,0,n^2d^2,-nd) \\
\end{equation*}	
and thus  $f^{-n}(\partial A_{\xi})$ are the set of all the points $(W_1,W_2,W_3,W_4)+(0,0,n^2d^2,-nd) $ such that
\begin{equation*}
\begin{split}
W_2&=W_1^2\\
W_3&=-\beta_2(1+\alpha_1+\alpha_1^2+\cdots+\alpha_1^{n-1})W_1+2d\beta_3(1+2\alpha_1+3\alpha_1^2+\cdots+n\alpha_1^{n-1})W_1        \\
W_4&=-\beta_3(1+\alpha_1+\alpha_1^2+\cdots+\alpha_1^{n-1})W_1.
\end{split} 
\end{equation*}	
Note that 
\begin{equation*}
\begin{split} 
g(t_n)f^{-1}(x,x^2,0,0)&=\begin{pmatrix}
e^{t_n\mu} & 0 & 0 & 0 \\
0 & e^{2t_n\mu} & 0 & 0 \\
0 & 0 & e^{2t_n\nu} & 0  \\
0 & 0 & 0 & e^{t_n\nu}
\end{pmatrix}
\begin{pmatrix}
\frac{1}{\alpha_1}x \\
\frac{1}{\alpha_1^2}x^2 \\
\frac{2d\beta_3-\beta_2}{\alpha_1}x+d^2\\
-\frac{\beta_3}{\alpha_1}x-d
\end{pmatrix}
\\
&=(Z_1,Z_2,Z_3,Z_4)+(0,0,n^2d^2,-nd) \\
\end{split}
\end{equation*}
where
\begin{equation*}
\begin{split} 
Z_2&=Z_1^2\\
Z_3&=(2d\beta_3-\beta_2)e^{t_n(2\nu-\mu)} Z_1\\
Z_4&=-\beta_3e^{t_n(\nu-\mu)} Z_1.
\end{split}
\end{equation*}}

{Suppose $\beta_3\neq 0$. Since $$g(t_n)f(A_{\xi})=f^n(A_{\xi}), \ g(t_n)f^{-1}(A_{\xi})=f^{-n}(A_{\xi})$$ by (\ref{fn}) and  (\ref{f-n}),  
$$e^{t_n(\nu-\mu)}=1+\frac{1}{\alpha_1}+\frac{1}{\alpha_1^2}+\cdots+\frac{1}{\alpha_1^{n-1}}$$
and 
$$e^{t_n(\nu-\mu)}=1+\alpha_1+\alpha_1^2+\cdots+\alpha_1^{n-1}$$ 
must hold simultaneously, which implies $\alpha_1=1$. But this  is impossible because 
$$e^{t_n(\nu-\mu)}=ne^{-t_n\mu}$$
and $1+\alpha_1+\alpha_1^2+\cdots+\alpha_1^{n-1}=n$ if $\alpha_1=1$.}

{$\beta_2=0$ is proved similarly, since if we suppose $\beta_2\neq 0$ then  
$$e^{t_n(2\nu-\mu)}=1+\frac{1}{\alpha_1}+\frac{1}{\alpha_1^2}+\cdots+\frac{1}{\alpha_1^{n-1}}$$
and 
$$e^{t_n(2\nu-\mu)}=1+\alpha_1+\alpha_1^2+\cdots+\alpha_1^{n-1}$$ 
must hold simultaneously, which is impossible.}

Up to now we have shown that if  $f_{t}$ is hyperbolic then
$$L_f=\begin{pmatrix}
\alpha_1 & 0 & 0 & 0 \\
0 &\alpha_1^2 & 0 & 0 \\
0 & 0 & 1 & 2d  \\
0 & 0 & 0 & 1
\end{pmatrix},$$
which implies that $A_{f(\xi)}$ and $A_{f^{-1}(\xi)}$ are just  translations of $A_{\xi}$. Since for any point $\zeta\neq \xi$ of $\partial B_{\xi}$,  $A_{\zeta}$ is either $g(t)(A_{f(\xi)})$ or $g(t)(A_{f^{-1}(\xi)})$
for some $t\in \mathbb R$, $A_{\zeta}$'s are all  translations of $A_{\xi}$.  Hence we can conclude that  $\Omega$ is affinely equivalent to $$\{(x_1,x_2,x_3,x_4) \in \mathbb R^4\,|\, x_2>x_1^{2}
,\ x_3>x_4^{2} \}$$ in this case.

\noindent {\bf Case 2. $f_{t}$ is  parabolic:} If $f_{t}$ is parabolic, then the linear part and the translation part of $f_t$  can be represented by 
$$L_{f_t}=\begin{pmatrix}
\alpha_1(t) & 0 & 0 & 0 \\
\beta_1(t) &\alpha_2(t)^2 & 0 & 0 \\
\beta_2(t) & 0 & 1 & 2t  \\
\beta_3(t) & 0 & 0 & 1
\end{pmatrix},\quad
t_{f_t}=\begin{pmatrix}
0 \\
0\\
t^2 \\
t
\end{pmatrix}.$$

Firstly, we show that   $\alpha_1(t)^2=\alpha_2(t)^2$ for all $n$:
	Since $f_{\theta^nt}(\xi)=g^n(f_t(\xi))=(0,0,\theta^{2n}t^2,\theta^nt)$, there is an element 
	$$h_n(t)=\begin{pmatrix}
	\delta_n(t) & 0 & 0 & 0 \\
	0 &\delta_n(t)^2 & 0 & 0 \\
	0 & 0 & \theta_n(t)^2 & 0  \\
	0 & 0 & 0 & \theta_n(t)
	\end{pmatrix} \in G_{\xi}$$
	such that $f_{\theta^nt}h_n(t)=g^nf_t$. Hence   we get   that for all $t\in \mathbb R$
	\begin{equation}\label{isotropy-eq}
				\alpha_1(\theta^nt)\delta_n(t)=\delta^n\alpha_1(t),\quad \alpha_2(\theta^nt)^2\delta_n(t)^2=\delta^{2n}\alpha_2(t)^2,\quad \theta_n(t)=\theta^{n}.
		\end{equation}
From $$\lim_{n \to \infty}\alpha_1(\theta^nt)=\alpha_1(0)\in \mathbb R^*,\quad  \lim_{n \to \infty}\alpha_2(\theta^nt)=\alpha_2(0)\in \mathbb R^*,$$ we get
	\begin{equation}\label{isotroy-f}
	\lim_{n \to \infty}\frac{\delta^n}{\delta_n(t)}=\frac{\alpha_1(0)}{\alpha_1(t)}=\frac{1}{\alpha_1(t)}, \quad \lim_{n \to \infty}\frac{\delta^{2n}}{\delta_n(t)^2}=\frac{\alpha_2(0)^2}{\alpha_2(t)^2}=\frac{1}{\alpha_2(t)^2}
		\end{equation}
		 and thus $\alpha_1(t)^2=\alpha_2(t)^2$ for all $t$.

Suppose that $\delta_n(t)=\delta_n(0)$ for all $t$ and all $n$. Then $\alpha_1(t)=\alpha_1(0)=1$ because 
	$$1=\frac{\alpha_1(0)}{\alpha_1(0)}=\lim_{n \to \infty}\frac{\delta^n}{\delta_n(0)}=\lim_{n \to \infty}\frac{\delta^n}{\delta_n(t)}=\frac{\alpha_1(0)}{\alpha_1(t)}$$
	by (\ref{isotroy-f}). That is, if 
	 	$$h_n(t)=\begin{pmatrix}
	\delta_n(0) & 0 & 0 & 0 \\
	0 &\delta_n(0)^2 & 0 & 0 \\
	0 & 0 & \theta^{2n} & 0  \\
	0 & 0 & 0 & \theta^n
	\end{pmatrix}$$
 for all $t$, then
$$L_{f_t}=\begin{pmatrix}
1 & 0 & 0 & 0 \\
\beta_1(t) & 1 & 0 & 0 \\
\beta_2(t) & 0 & 1 & 2t  \\
\beta_3(t) & 0 & 0 & 1
\end{pmatrix}.$$	
So from the fact that $f_{t+t'}=f_tf_{t'}$, we get 
\begin{lemma}
$$\beta_1(t)=t\beta_1(1),\quad \beta_3(t)=t\beta_3(1). $$
\end{lemma}
\begin{proof}$f_{t+t'}=f_tf_{t'}$ implies that
$$\beta_1(t)+\beta_1(s)=\beta_1(t+s),\beta_1(0)=0,$$
$$\beta_3(t)+\beta_3(s)=\beta_3(t+s),\beta_3(0)=0.$$ We show that $\beta_1$ and $\beta_3$ are linear.
Note that
$$q\beta_1(t)=\beta_1(p\frac{q}{p}t)=p\beta_1(\frac{q}{p}t),\ \beta_1(t)+\beta_1(-t)=\beta_1(0)=0.$$
Hence $\beta_1(\frac{q}{p}t)=\frac{q}{p}\beta_1(t)$ and $\beta_1(rt)=r\beta_1(t)$ for any rational number $r$. By continuity of $\beta_1$ and the density of rational numbers in real numbers implies 
$$\beta_1(Rt)=R\beta_1(t)$$ for any real number $R$, hence $\beta_1(t)=t\beta_1(1)$. The same argument gives $\beta_3(t)=t\beta_3(1)$.
\end{proof}
Since $g^nf_1(\xi)=f_{\theta^n}(\xi)$, the face $g^nf_1(A_{\xi})$ equals the face $f_{\theta^n}(A_{\xi})$. 
But
\begin{equation*} 
\begin{split}
g^nf_1(\partial A_{\xi})&=\{g^nf_1(x,x^2,0,0)\,|\, x\in \mathbb R\}\\
&=\{(\delta^nx,\delta^{2n}x^2+\delta^{2n}\beta_1(1)x,\theta^{2n}\beta_2(1)x,\theta^n\beta_3(1)x)\,|\, X\in \mathbb R\}+(0,0,\theta^{2n},\theta^n) \\
&=\{(X,X^2+\delta^{n}\beta_1(1)X,\frac{\theta^{2n}}{\delta^{n}}\beta_2(1)X,\frac{\theta^n}{\delta^{n}}\beta_3(1)X)\,|\, X\in \mathbb  R\}+(0,0,\theta^{2n},\theta^n) \\
\end{split}
\end{equation*}	
and
\begin{equation*} 
\begin{split}
f_{\theta^n}(\partial A_{\xi})&=\{f_{\theta^n}(x,x^2,0,0)\,|\, x\in \mathbb R \}\\
&=\{(x,x^2+\beta_1(\theta^n)x,\beta_2(\theta^n)x,\beta_3(\theta^n)x)\,|\, x\in \mathbb R\}+(0,0,\theta^{2n},\theta^n)\\
&=\{(x,x^2+\theta^n\beta_1(1)x,\beta_2(\theta^n)x,\theta^n\beta_3(1)x)\,|\, x\in \mathbb R\}+(0,0,\theta^{2n},\theta^n).
\end{split}
\end{equation*}	
This implies $\delta=\theta=1$, 
which is a contradiction. 

Now we have a non-constant continuous function $\delta_n(t)$ from $\mathbb R$ to $\mathbb R$ for some $n$. 
 So there is an element $h_n(t_0)\neq g^n$ of $G_{\xi}$  such that $\delta_n(t_0)\neq \delta^n$ and $$h_n(t_0)=\begin{pmatrix}
\delta_n(t_0) & 0 & 0 & 0 \\
0 &\delta_n(t_0)^2 & 0 & 0 \\
0 & 0 & \theta^{2n} & 0  \\
0 & 0 & 0 & \theta^n
\end{pmatrix} $$ 
by (\ref{isotropy-eq}),
and thus we get an element $h=g^{-n}h_n(t_0)\in G_{\xi}$ such that 
$$h=\begin{pmatrix}
\delta^{-n}\delta_n(t_0) & 0 & 0 & 0 \\
0 &\delta^{-2n}\delta_n(t_0)^2 & 0 & 0 \\
0 & 0 & 1 & 0  \\
0 & 0 & 0 & 1
\end{pmatrix}\neq \begin{pmatrix}
1 & 0 & 0 & 0 \\
0 &1 & 0 & 0 \\
0 & 0 & 1 & 0  \\
0 & 0 & 0 & 1
\end{pmatrix}.  $$ 
From the existence of such an element of $G_{\xi}$, using a similar argument as before  we can conclude that $\Omega$ is affinely equivalent to $$\{(x_1,x_2,x_3,x_4) \in \mathbb R^4\,|\, x_2>x_1^{2}
,\ x_3>x_4^{2} \}.$$
\end{proof}

\begin{Prop}\label{dim4-asy4}Let $\Omega$ be a properly convex
quasi-homogeneous affine domain in $\mathbb R^4$ with
4-dimensional asymptotic cone. Then $\Omega $ is a  cone and
affinely equivalent to one of the following:
\begin{enumerate}
\item[\rm(i)]an elliptic cone,

\item[\rm(ii)]a non-elliptic strictly convex cone,
\item[\rm(iii)] a double cone over a triangle, i.e.,$$\{(x_1,x_2,x_3,x_4) \in \mathbb R^4\,|\, x_i>0 \text{ for }
i=1,2,3,4 \},$$
\item[\rm(iv)]a double cone over an ellipse,
\item[\rm(v)]a double cone over a non-elliptic strictly convex
domain,
\item[\rm(vi)]a cone over a 3-dimensional non-strictly convex
indecomposable projective domain.
\end{enumerate}
\end{Prop}
\begin{proof}
By Vey \cite{V3} or Theorem \ref{thm-folliation}, $\Omega$ is a cone onto a 3-dimensional properly convex quasi-homogeneous projective domain $P\Omega$ in $\mathbb {RP}^3$. If $P\Omega$ is strictly convex, then $\Omega$ is either (i) or (ii). If $P\Omega$ has a 2-dimensional face $F$, then $F$ is a conic face of $P\Omega$ and thus quasi-homogeneous projective domain. Since $F$ is an ellipse if it is strictly convex and its boundary is twice differentiable and $F$ is a triangle if it is not strictly convex, $\Omega$ is (iii) or (iv) or (v). (vi) is the case when $\Omega$ is neither strictly convex nor has no 2-dimensional face.
\end{proof}
Cones in (vi) are all actually divisible cones, which is immediate from
the following proposition.
\begin{Prop}\label{indecNonstr3}Let $\Omega$ be a properly convex
quasi-homogeneous projective domain in $\mathbb {RP}^3$ which is
indecomposable and non-strictly convex. Then
$\Aut_\text{proj}(\Omega)$ is irreducible and discrete.
\end{Prop}
\begin{proof}
Suppose $\Aut_\text{proj}(\Omega)$ is reducible and $L$ is a
stable projective subspace of $\mathbb {RP}^3$. Then
$$L\cap\Omega=\emptyset \text{ and }
L\cap\overline\Omega\neq\emptyset,$$ by Vey \cite{V3}. Since
$\Omega$ has no 2-dimensional face, $L\cap\overline\Omega$ is a
point or a closed line segment.

If $L\cap\overline\Omega$ is a closed line segment $l$, then we
can choose a one dimensional projective subspace $L'$ such that
$L'\cap\overline\Omega$ is a line segment $l'$ such that
$l\cap\l'=\emptyset$ : Since $\Omega$ is not strictly convex,
there are infinitely many 1-dimensional faces. They cannot
intersect in their interior because $\Omega$ cannot have
2-dimensional face and only two faces can meet at their end points
because any extreme point cannot be a conic point by the
indecomposability of $\Omega$. Now we consider a sequence of
projective transformation $\{g_i\}$ in $\Aut_\text{proj}(\Omega)$
which converges to a singular projective transformation $g$ whose
range $R(g)$ is $L'$. By stability of $L$ and
$l\cap\l'=\emptyset$, the kernel $K(g)$ must be $L$, which implies
that $l$ is a conic face of $\Omega$ by Lemma \ref{supker}. This contradicts that
$\Omega$ is indecomposable.

So $L\cap\overline\Omega$ is a point $\xi$ in $\partial\Omega$.
Similarly we can find infinitely many maximal closed line segments
which do not contain $\xi$. This time we consider a sequence of
projective transformation $\{g_i\}$ in $\Aut_\text{proj}(\Omega)$
which converges to a singular projective transformation $g$ whose
range $R(g)$ is $\{\xi\}$. Since $\xi$ is a fixed point, every
maximal closed line segment which is disjoint from $\{\xi\}$ should
be in the kernel $K(g)$, which is a contradiction.

Up to now we've proved that $\Aut_\text{proj}(\Omega)$ is
irreducible. The discreteness of $\Aut_\text{proj}(\Omega)$
follows immediately from the irreducibility of
$\Aut_\text{proj}(\Omega)$ by Proposition 4.2 of \cite{Bt}, since $\Omega$ is not
homogeneous.
\end{proof}

Indecomposable non-strictly convex projective divisible domains in
$\mathbb {RP}^3$ were studied by Y. Benoist in \cite{Bt4}.

\section{remarks}
We have seen in the previous section that the Markus conjecture is true for convex affine manifolds when the dimension is $\leq 5$.  Theorem \ref{Markus} implies that if  $\Lambda_{\Aut^0(\Omega)}\neq \emptyset$ for any quasi-homogeneous domain $\Omega \neq  \br^n$, then Markus conjecture is completely solved in convex case. So far, we have the following observation.
\begin{Prop}\label{noncpt}  Let $\Omega$ be a properly convex
quasi-homogeneous domain in  $\br^n$. Let $G^0$ be the identity
component of $G=\Aut(\Omega)$. Then $G^0$ is noncompact.
\end{Prop}
\begin{proof}
If $G^0$ is compact, it will have a fixed point $x_0$ in $\Omega$.
Since $G^0$ is normal in $G$, $G^0$ will fix $Gx_0$ pointwise.
This implies that $G^0$ is trivial because $CH(Gx_0)=\Omega$ by Proposition \ref{Vey}, so
$\Aut(\Omega)$ is a discrete group. But $\Aut(\Omega)$ cannot be
discrete because every properly convex divisible affine domain is a cone
by Theorem \ref{VeyThm} and the automorphism group of a cone is not discrete. So
we conclude that $G^0$ is noncompact.

\end{proof}

\section*{Acknowledgements }
This work was done at Korea Institute for Advanced Study (KIAS) while 
 the first author was a visiting professor in 2016-2017. The first author  is grateful for the warm hospitality during her stay. 
 Both authors thank an anonymous referee for several insightful questions.


\begin{thebibliography}{99}


\bibitem{Bt}Y. Benoist, \textit{Convex divisible II}, Duke Math. J. \textbf{120} (2003), no.1, 97-120.

\bibitem{Bt4}Y. Benoist, \textit{Convex divisible IV: Structure du bord en dimension 3}, Invent. math. \textbf{164} (2006), 249-278.

\bibitem{Ben}J. P. Benz\'{e}cri, \textit{Sur les varietes localement
affines et projectives}, Bull. Soc. Math. Fr. \textbf{88} (1960), 229-332.



\bibitem{C} Y. V. Carri\`{e}re, \textit{Autour de la conjecture de
L. Markus sur les vari\'{e}t\'{e}s affines}, Invent. Math. \textbf{95}
(1989), 615--628.


\bibitem{CLT} D. Cooper, D. Long and S. Tillmann, \textit{On convex projective manifolds and cusps}, Adv. Math. \textbf{277(4)}
(2015), 181--251.

\bibitem{Fr}D. Fried, \textit{Distality, completeness and affine structures}, J. Diff. Geom., \textbf{24} (1986), 265-273.

\bibitem{FG} D. Fried, W.  M. Goldman and M. W. Hirsch, \textit{Affine manifolds with nilpotent
holonomy}, Comment. Math. Helv. \textbf{56} (1981), 487--523.

\bibitem{G} W. M. Goldman, \textit{Two examples of affine manifolds}, Pacific J. Math. \textbf{94} (1981), no. 2, 327--330.

\bibitem{GH1}W.  M. Goldman and M. W. Hirsch, \textit{The radiance obstruction and parallel forms on affine manifolds}, Trans. Amer. Math. Soc. \textbf{286} (1984), 629-649.


\bibitem{GH2} W. M. Goldman and M. W. Hirsch, \textit{Affine
manifolds and orbits of algebraic groups}, Trans. Amer. Math. Soc.
\textbf{295} (1986), no. 1, 175--198.

\bibitem{Jo} K. Jo, \textit{Quasi-homogeneous domains and convex affine manifolds}, Topology Appl. \textbf{134} (2003), no. 2, 123--146.

\bibitem{Jo5} K. Jo, \textit{Homogeneity, Quasi-homogeneity and differentiability of
domains}, Proc. Japan Acad. Ser. A Math. Sci. \textbf{79} (2003),
no, 9, 150-153

\bibitem{Jo1} K. Jo, \textit{Differentiability of quasi-homogeneous convex affine domains}, J. Korean Math. Soc. \textbf{42} (2005), No. 3, 485-498.


\bibitem{Jo4} K. Jo, \textit{A rigidity result for domains with a locally strictly convex point}, Adv. Geom. \textbf{8} (2008), No. 3, 315-328.


\bibitem{Jo2} K. Jo, \textit{Asymptotic foliation of quasi-homogeneous convex affine domains}, Commun. Korean Math. Soc. \textbf{32} (2017), No. 1, 629-649.

\bibitem{JK} K. Jo and I. Kim, \textit{Convex affine domains and Markus Conjecture}, Math. Z. \textbf{248} (2004), no. 1, 173--182.

\bibitem{Ko} J. L. Koszul, \textit{Deformation des connexions
localement plats}, Ann. Inst. Fourier \textbf{18} (1968),
103--114.

\bibitem{Ku} N. H. Kuiper, \textit{On convex locally projective
spaces}, Convegno Intern. Geom. Diff. Italy (1953), 200--213.


\bibitem{Ma} L. Markus, \textit{Cosmological models in
differential geometry}, Mimeographed notes, Univ. of Minnesota,
1962, p. 58.

\bibitem{Marquis} L. Marquis, \textit{Around groups in Hilbert geometry}, In \textit{Handbook of Hilbert geometry}, IRMA Lectures in Mathematics and Theoretical Physics Vol. 22, 207--261.


\bibitem{Rogers} C. A. Rogers, \textit{Some problems in the Geometry of convex bodies}, The Geometric Vein:The Coxeter Festschrift, Springer-Verlag (1981), 279--284.

\bibitem{Ru} W. Rudin, \textit{Functional analysis}, International
 series in Pure and Applied Mathematics, p.75.


\bibitem{V1} J. Vey, \textit{Une notion d'hyperbolicite sur les
varietes localement plates}, C. R. Acad. Sci. Paris Ser. AB
\textbf{266} (1968), A622--A624.

\bibitem{V2} J. Vey, \textit{Sur les automorphismes affines des
ouverts convexes dans les espaces numeriques}, C. R. Acad. Sci.
Paris Ser. AB \textbf{270} (1970), A249--A251.

\bibitem{V3} J. Vey, \textit{Sur les automorphismes affines des
ouverts convexes saillants}, Ann. Scuola Norm. Sup. Pisa
\textbf{24} (3) (1970), 641--665.

\bibitem{VK} E. B. Vinberg and V. G. Kats,
\textit{Quasi-homogeneous cones}, Math. Notes \textbf{1} (1967),
231--235. (translated from Math. Zametki \textbf{1} (1967),
347--354)


\bibitem{VK2} E. B. Vinberg,
\textit{The theory of convex homogeneous cones}, trans. Moscow math. Soc. \textbf{12} (1963),
340--403.


\end{thebibliography}
\end{document}